\documentclass{amsart}
\usepackage[utf8]{inputenc}
\usepackage{verbatim}

\usepackage{amssymb}
\usepackage{soul,xcolor}
\setstcolor{red}

\usepackage{latexsym,amssymb,amsmath}
\usepackage{epsfig}
\input xy
\xyoption{all}
\usepackage[english]{babel}

\usepackage{amsmath}
\usepackage{bbm} 
\newtheorem{theorem}{Theorem}[section]
\newtheorem{lemma}[theorem]{Lemma}
\newtheorem{prop}[theorem]{Proposition}
\newtheorem{coro}[theorem]{Corollary}
\newtheorem{definition}[theorem]{Definition}
\newtheorem{remark}[theorem]{Remark}

\def\PP{\mathbb{P}}\def\AA{\mathbb{A}}\def\RR{\mathbb{R}}
\def\CC{\mathbb{C}}\def\HH{\mathbb{H}}\def\OO{\mathbb{O}}
\def\ZZ{\mathbb{Z}}\def\QQ{\mathbb{Q}}
 
\def\cA{{\mathcal A}}\def\cZ{{\mathcal Z}}

\def\cE{{\mathcal E}}

\def\cG{{\mathcal G}}

\def\cJ{{\mathcal J}}

\def\cQ{{\mathcal Q}}
\def\cM{{\mathcal M}}
\def\cO{{\mathcal O}}
\def\cS{{\mathcal S}}
\def\cU{{\mathcal U}}
\def\cX{{\mathcal X}}
\def\ra{\rightarrow}\def\lra{\longrightarrow}
\def\fg{{\mathfrak g}}\def\fh{{\mathfrak h}}
\def\fso{{\mathfrak{so}}}\def\ft{{\mathfrak{t}}}\def\fsl{{\mathfrak{sl}}}\def\fsp{{\mathfrak sp}}

\def\fe{{\mathfrak e}}\def\ff{{\mathfrak f}}
\def\fk{{\mathfrak k}}\def\fg{{\mathfrak g}}
\def\fj{{\mathfrak j}}\def\fp{{\mathfrak p}}
\def\fS{{\mathfrak S}}
\def\fD{{\mathfrak D}}
\DeclareMathOperator{\Ho}{H}
\DeclareMathOperator{\Aut}{Aut}

\DeclareMathOperator{\SL}{SL}
\DeclareMathOperator{\GL}{GL}
\DeclareMathOperator{\PGL}{PGL}
\DeclareMathOperator{\PSL}{PSL}
\DeclareMathOperator{\SO}{SO}
\DeclareMathOperator{\PSO}{PSO}
\DeclareMathOperator{\Sp}{Sp}
\DeclareMathOperator{\Stab}{Stab}
\DeclareMathOperator{\id}{id}
\DeclareMathOperator{\ad}{ad}
\DeclareMathOperator{\Ad}{Ad}
\DeclareMathOperator{\Fl}{Fl}
\DeclareMathOperator{\OGr}{OGr}
\DeclareMathOperator{\IG}{IG}

\DeclareMathOperator{\OG}{OG}
\DeclareMathOperator{\OFl}{OFl}
\DeclareMathOperator{\rank}{rank}
\DeclareMathOperator{\diag}{diag}
\DeclareMathOperator{\Hilb}{Hilb}
\DeclareMathOperator{\Lie}{Lie}
\DeclareMathOperator{\Out}{Out}
\DeclareMathOperator{\pt}{pt}
\def\GG{\mathbb{G}}

\usepackage{xcolor}
\def\vlad#1{\textcolor{blue}{{\bf Vlad:} #1 {\bf }}}
\def\laurent#1{\textcolor{red}{{\bf *** Laurent:} #1 {\bf ***}}}

\author{Vladimiro Benedetti}
\author{Laurent Manivel}

\title[Automorphisms of hyperplane sections]{On the automorphisms of hyperplane sections \linebreak of generalized Grassmannians}

\begin{document}

\maketitle

\begin{abstract}
Given a smooth hyperplane section $H$ of a rational homogeneous space $G/P$ with Picard number one, we address the 
question whether it is always possible to lift an automorphism of $H$ to the Lie group $G$, or more precisely to 
$\Aut(G/P)$. Using linear spaces and quadrics in $H$, we show that the answer is positive up to a few well
understood exceptions related to Jordan algebras. When $G/P$ is an adjoint variety, we show how to describe 
$\Aut(H)$ completely, extending results obtained by Prokhorov and Zaidenberg when $G$ is the exceptional group $G_2$. 
\end{abstract}

\section{Introduction}
In complex projective geometry, rational homogeneous spaces and their linear sections are an important source of 
easily available Fano manifolds. Already in dimension three they play a major role in the Fano-Iskovskikh classification 
of prime Fano threefolds \cite{ip}. In genus seven to ten these are indeed linear sections of certain specific generalized
Grassmannians (rational homogeneous spaces of Picard number one). 

In the study of such families of linear sections, a natural problem is to understand the automorphism groups. Since we 
start with highly symmetric varieties, these groups will remain big in low codimension. 
Describing the generic automorphism group of the family can already be challenging \cite{dm}, but these groups can in fact vary in quite
intriguing, somewhat erratic ways. Interesting phenomena have been observed in connection with the existence 
problem for K\"ahler-Einstein metrics, or with rigidity questions  
\cite{fh,bfm}. For Mukai fourfolds of genus $10$, a complete treatment was recently given in \cite{pz1,pz2}, where it is 
proved that exactly four types of automorphism groups can occur for hyperplane sections of the adjoint variety 
of type $G_2$, a very interesting generalized Grassmannian of dimension five and index three. 

For these hyperplane sections, a nice argument due to Mukai \cite{mukai-K3Fano}, and based on the special properties of K3 surfaces, 
implies that any automorphism can be lifted to $G_2$, which essentially reduces the problem to a Lie theoretic 
question. In \cite{pz1}, the authors asked about the same question for a smooth hyperplane section of an
arbitrary generalized 
Grassmannian $G/P$: is it true that any element of the connected component of the automorphism group can be lifted to $G$? 
Actually one can also ask: when is it true  that any element  of the automorphism group can be lifted to $G$?
The first goal of this paper is to give a complete answer to both questions.

\begin{theorem} 
\label{thm_1.1}
Let $H$ be a smooth hyperplane section of a generalized Grassmannian $X=G/P$ with respect to its minimal 
$G$-equivariant embedding. Any automorphism of $H$ in $\Aut^0(H)$ can be extended to an automorphism of $X$, i.e. $\Aut^0(H)\subset \Aut(X)$. This is true for any automorphism of $H$ as well, except possibly when $X$ is already a hyperplane section of  another 
generalized Grassmannian.
\end{theorem}

Apart from the trivial cases of projective spaces and quadrics, 
the latter situation occurs exactly when $X=F_4/P_4$, or when $X=\IG(2,2n)$ is a Grassmannian of isotropic planes
with respect to some symplectic form. 
We give a complete treatment of these varieties, which have interesting connections with Jordan algebras and are also known
as coadjoint (but not adjoint) varieties. To give a sample of our results, we prove that for $H$ a smooth hyperplane section of $\IG(2,2n)$, the component group of 
$\Aut(H)$ can be an icosahedral group when $n$ is any even number such that $n> 10$ and 
$$n=0, \; 2, \; 12, \; 20, \; 30, \; 32, \; 42, \; 50 \quad \mathrm{modulo}\;  60.$$
Moreover, when this happens, there exist some automorphisms of $H$ that cannot be lifted to the symplectic group $G=\Sp_{2n}$ (we refer to Theorem \ref{thm_exact_sequence_aut_jordan}, the remark following it and Proposition \ref{lem_unnatural_aut_jordan} for more precise statements). 

Our methods are heavily based on the study of linear spaces and quadrics of high dimension contained in $G/P$ and its 
hyperplane sections. Of course this natural approach has been used before in closely related context, see for example 
\cite{mihai} and more recently \cite{kps}, where lines and conics are considered systematically. In our set-up 
the interest of rather considering linear spaces and quadrics of maximal or submaximal dimension (that is, maximal
minus one) stems from the useful Property we denoted (UE). This is simply the fact, not always but very often verified 
in our setting, that linear spaces and quadrics of submaximal dimension admit a unique extension to maximal dimension. 
Once this property is established, the lifting property for automorphisms of smooth hyperplane sections is in 
most cases easily confirmed.

\medskip
An important issue in the study of their automorphisms is that in general we are not able to describe the open set parametrizing smooth 
hyperplane sections of $G/P$, or equivalently the complement to this open set, which is the projective dual to 
$G/P$ (see \cite{tevelev} for an overview of what is known on this question). A remarkable exception is when $G/P$ 
is the {\it adjoint variety} of the simple Lie algebra $\fg=\Lie(G)$, that is,
the closed $G$-orbit inside $\PP (\fg)$ (or the projectivization of the minimal nontrivial nilpotent orbit). Then
an equation of the projective dual was obtained by Tevelev \cite{tevelev}, which for $G$ simply laced implies that the smooth 
hyperplane sections are exactly those that are defined by regular semisimple elements; while the situation 
is richer, and more complicated, for $G$ non simply laced (typically for the aforementioned case of $G_2$). 
We give a complete treatment of this case and describe all the possible automorphism groups of the smooth 
hyperplane sections. The main result is the following, where $R$ denotes the root system of $\fg$. 

\begin{theorem} 
\label{thm_intro_adjoint}
Let $H_x$ be a smooth hyperplane section of the adjoint variety of $\fg$, defined by the hyperplane 
Killing-orthogonal to $x\in\fg$. Let $x=x_s+x_n$ be the Jordan decomposition of $x$, where the semisimple 
part $x_s$ can be supposed to belong to a fixed Cartan subalgebra $\fh$ of $\fg$. Then there exists 
a subroot system $R^\perp$ of $R$, depending on $x$, of corank at most two, with Weyl group $W^\perp$,
such that $$\Aut(H_x)/N_G^0([x])=\Stab_{W^\perp}([x_s])\rtimes D_x,$$
where $D_x$ is contained inside the outer automorphism group $\Out(R^\perp)$ and 
is isomorphic either to $1, \ZZ_2$ or $\cS_3$.
\end{theorem}

Here $N_G^0([x])$ denotes the connected component of the normalizer in $G$ of the line $[x]$ and can easily 
be described; in most cases this is just a maximal torus. 
With respect to $R^\perp$, $x_s$ is always regular and the stabilizer $\Stab_{W^\perp}([x_s])$ can be 
deduced from classical results of Springer \cite{Springer}. Even if we restrict to $x$ regular semisimple, 
it turns out that the
finite part of the automorphism group can vary in a rather subtle way. For a complete description of the possible automorphism groups of hyperplane sections of adjoint varieties, we refer to Table \ref{final_table}.

\smallskip
Finally, let us point out that the last two classes of hyperplane sections (related to Jordan algebras or contained in
adjoint varieties) constitute an interesting testing ground for Dubrovin's conjecture in mirror symmetry, which is now
known to be true for many rational homogeneous spaces but not much beyond that. Indeed, the fact that these hyperplane 
sections have a big automorphism group (and admit in particular an action of a \emph{big} torus) allows us to use equivariant and localization techniques in order to understand their cohomology. A recent attempt to test Dubrovin's conjecture, 
starting from the (quantum) cohomology side, can be found in \cite{beneperrin}.


\medskip\noindent {\it Acknowledgements.} We would like to thank Mikhail Zaidenberg for his useful comments on a preliminary version of the paper. 
We would also like to express our gratitude to Alexander Kuznetsov for his careful and insightful reading,  which led to significant improvements.

We acknowledge support from the ANR project FanoHK, grant
ANR-20-CE40-0023.

\section{Preliminaries} 

In this section we recall a few basic facts about  generalized Grassmannians and their linear spaces, mostly extracted from \cite{lm-proj}. First recall that simple complex Lie algebras are classified by (connected) Dynkin diagrams, whose nodes we 
number as in \cite{bourbaki}. Simple Lie algebras can be upgraded to simple Lie groups, whose actions on projective varieties 
are of interest. The following definition is classical. 

\begin{definition} A generalized Grassmannian is a rational complex projective homogeneous variety with Picard number one.
\end{definition} 

Another classical fact is that generalized Grassmannians are classified by Dynkin diagrams with one marked node.  This 
has to be taken with a grain of salt: a generalized Grassmannian can always be described as $X=G/P$, where $P$ is a maximal 
parabolic subgroup of a simple Lie group $G$; then $P$ can be defined (uniquely, up to conjugation) by the choice of a
node on the Dynkin diagram of $\Lie(G)$. However it sometimes happens (in the {\it exceptional cases}, according to the slightly 
confusing terminology used in \cite{demazure}) that $X=G'/P'$ with $\Lie(G')\ne \Lie(G)$. In order to avoid
such inconveniences, we will always suppose that $G$ coincides with $\Aut(X)$, at least up to a finite group. 

By the way note that $\Aut(X)$ can be described explicitly. Recall that $\Aut(G)$, say for $G$ adjoint, 
contains $G$ as the normal subgroup of inner automorphisms, while $\Out(G)=\Aut(G)/G$ can be identified with
the symmetry group of the Dynkin diagram. The following extension is due to Demazure \cite[Proposition 1]{demazure}; it holds for any complex projective rational  homogeneous variety  (the corresponding result for the connected group of automorphisms was already contained in \cite{oniscik}).

\begin{theorem}
\label{thm_Demazure}
Suppose that $X=G/P$ is not isomorphic to any $G'/P'$ with $\Lie(G')$ bigger than $\Lie(G)$. Then 
$$\Aut(X)=\Aut(G,X)\subset \Aut(G)$$
is the preimage of the subgroup of $\Out(G)$ acting on the Dynkin diagram by fixing the nodes that 
define $X$. 
\end{theorem} 

\medskip
As we have seen, nodes of Dynkin diagrams are (essentially) in correspondence with generalized Grassmannians. 
More classically, 
they are also in correspondence with simple roots, which are either all of the same length (in the ADE, or simply laced 
types), or of two possible lengths (in the BCFG types). We will therefore distinguish generalized Grassmannians of 
long (including all the simply laced cases) or short type. 

When one is interested in linear spaces, as we will be in the sequel, this makes an essential difference. Indeed, recall that any generalized Grassmannian $X=G/P$  admits a 
minimal $G$-equivariant embedding (defined by the ample generator of its Picard group, and generalizing the Pl\"ucker
embedding of an ordinary Grassmannian) inside the projectivization $\PP (V)$ of a simple (even fundamental) $G$-module. 
We can therefore consider linear subspaces with respect to this minimal embedding. The following two statements, with
more details, can be 
found in \cite{lm-proj}.

\begin{theorem}
Projective lines in a generalized Grassmannian of long type $X=G/P$ are parametrized by a $G$-homogeneous 
variety. When $X=G/P$ is a generalized Grassmannian of short type, the variety of projective lines contained in it has two $G$-orbits. 
\end{theorem} 

In the long case, one can push the analysis further and describe all the linear spaces. 
Recall first that $\PP^n=\SL_{n+1}/P$ is defined, as a generalized Grassmannian, by a Dynkin diagram $A_n$ 
marked by  one of its extremal nodes. Inside the marked Dynkin diagram defining $G/P$, one can 
cut out such a diagram by erasing some finite set $S$ of nodes, which in turn defines a parabolic
subgroup $P_S$ of $G$. Tits' theory of {\it shadows} implies that $G/P_S$ parametrizes a family 
of $\PP^n$'s on $G/P$. 

 Let us recall the general construction of shadows in \cite{tits}. Consider a set $S'$ of nodes, disjoint from $S$, and denote by $D'$ the subdiagram of $D$ from which the nodes of $S'$ have been erased (as well as the edges connecting them to other nodes). The variety $G/P_{S \cup S'}$ projects to both $G/P_S$ and $G/P_{S'}$. The fibers of the projection $G/P_{S\cup S'}\to G/P_{S}$ are encoded in the Dynkin diagram: they are homogeneous varieties $H/Q_{S'}$, where $H$ is a semisimple Lie group with Dynkin diagram $D'$, and $Q_{S'}\subset H$ is a 
 parabolic subgroup defined by $S'\subset D'$. Thus, $G/P_{S}$ parametrizes a family of subvarieties of $G/P_{S'}$ which are isomorphic to $H/Q_{S'}$.

\begin{theorem}
\label{thm_subdiagrams_An}
In the long case, the copies of $\PP^n$ inside the generalized Grassmannian $X=G/P$ are parametrized by a 
finite union, that we denote $\Hilb_{\PP^n}(G/P)$, of $G$-homogeneous varieties $G/P_S$, where $S$ is such 
that its complement inside the 
marked Dynkin diagram of $X$ contains the marked Dynkin diagram of $\PP^n$ as a connected component. 
\end{theorem}  

Throughout the paper, for an embedded projective variety $T\subset\PP^N$, we use the (sloppy) notation 
$\Hilb_{\PP^n}(T)$ for the closed subvariety of the Grassmannian parametrizing copies of $\PP^n $ linearly embedded in $T$.

\medskip\noindent {\it Example}. Let us illustrate this procedure for the so called Cayley plane $X=E_6/P_1$. Applying the previous recipe we get 
$$\begin{array}{lll} 
\Hilb_{\PP^1}(X)=E_6/P_3, & \Hilb_{\PP^2}(X)=E_6/P_4, & \Hilb_{\PP^3}(X)=E_6/P_{2,4}, \\
\Hilb_{\PP^4}(X)=E_6/P_5\cup E_6/P_{2,6}, & \Hilb_{\PP^5}(X)=E_6/P_2, & \Hilb_{\PP^6}(X)=\emptyset . 
\end{array}$$
This can be checked by chasing the marked subdiagrams of type $(A_n,\alpha_1)$ inside $(E_6,\alpha_1)$ ($\alpha_i$ being the $i$th simple root of the corresponding Dynkin diagram in Bourbaki's convention), 
which are the following ones:

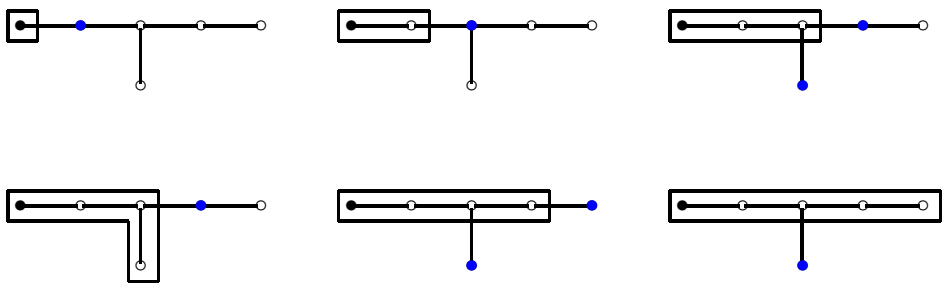
\setlength{\unitlength}{4mm}
\thicklines
\begin{picture}(18,5.9)(.5,-1.5)
\multiput(2,2)(2,0){4}{$\circ$}\put(4,0){$\circ$}
\put(0,2){$\bullet$}
\multiput(0.3,2.2)(2,0){4}{\line(1,0){1.8}}
\put(4.2,0.3){\line(0,1){1.8}}
\multiput(-.2,1.7)(1,0){2}{\line(0,1){1}}
\multiput(-.2,1.7)(0,1){2}{\line(1,0){1}}

\multiput(13,2)(2,0){4}{$\circ$}\put(15,0){$\circ$}
\put(11,2){$\bullet$}
\multiput(11.3,2.2)(2,0){4}{\line(1,0){1.8}}
\put(15.2,0.3){\line(0,1){1.8}}
\multiput(10.8,1.7)(3,0){2}{\line(0,1){1}}
\multiput(10.8,1.7)(0,1){2}{\line(1,0){3}}

\multiput(24,2)(2,0){4}{$\circ$}\put(26,0){$\circ$}
\put(22,2){$\bullet$}
\multiput(22.3,2.2)(2,0){4}{\line(1,0){1.8}}
\put(26.2,0.3){\line(0,1){1.8}}
\multiput(21.8,1.7)(5,0){2}{\line(0,1){1}}
\multiput(21.8,1.7)(0,1){2}{\line(1,0){5}}

\color{blue}
\put(2,2){$\bullet$}
\put(15,2){$\bullet$}
\put(28,2){$\bullet$}
\put(26,0){$\bullet$}
\end{picture}

\begin{picture}(18,5.9)(.5,-1.5)
\multiput(2,2)(2,0){4}{$\circ$}\put(4,0){$\circ$}
\put(0,2){$\bullet$}
\multiput(0.3,2.2)(2,0){4}{\line(1,0){1.8}}
\put(4.2,0.3){\line(0,1){1.8}}
\put(-.2,1.7){\line(0,1){1}}
\put(-.2,1.7){\line(1,0){4}}
\put(-.2,2.7){\line(1,0){5}}
\put(4.8,-.3){\line(0,1){3}}
\put(3.8,-.3){\line(0,1){2}}
\put(3.8,-.3){\line(1,0){1}}

\multiput(13,2)(2,0){4}{$\circ$}\put(15,0){$\circ$}
\put(11,2){$\bullet$}
\multiput(11.3,2.2)(2,0){4}{\line(1,0){1.8}}
\put(15.2,0.3){\line(0,1){1.8}}
\multiput(10.8,1.7)(7,0){2}{\line(0,1){1}}
\multiput(10.8,1.7)(0,1){2}{\line(1,0){7}}

\multiput(24,2)(2,0){4}{$\circ$}\put(26,0){$\circ$}
\put(22,2){$\bullet$}
\multiput(22.3,2.2)(2,0){4}{\line(1,0){1.8}}
\put(26.2,0.3){\line(0,1){1.8}}
\multiput(21.8,1.7)(9,0){2}{\line(0,1){1}}
\multiput(21.8,1.7)(0,1){2}{\line(1,0){9}}

\color{blue}
\put(6,2){$\bullet$}
\put(19,2){$\bullet$}
\put(15,0){$\bullet$}
\put(26,0){$\bullet$}
\end{picture}

The nodes defining the parameter space in each case have been marked in blue.

\medskip
The short case is more complicated and a uniform description is lacking. Note that in type $B$ there is only one
short classical Grassmannian, the orthogonal Grassmannian $\OG(n,2n+1)=\SO_{2n+1}/P_n$; in fact this is one of the 
exceptional cases in the sense of Demazure, since it is isomorphic to a spinor variety, that is, one of the 
two connected components of $\OG(n+1,2n+2)$, which in turn is a generalized Grassmannian of type $D_{n+1}$. 
In type $G_2$ there are only two generalized Grassmannians, both of dimension five; one of them is actually a quadric,
and the other one is of long type. 

So we really remain only with the generalized Grassmannians in type $C$, that is the symplectic Grassmannians 
$\IG(k,2n)$, which are of short type for $k<n$; and the two generalized Grassmannians of short type for $F_4$,
namely $F_4/P_3$ and $F_4/P_4$.

\section{General strategy}\label{strategy}

\subsection{Long roots} 
\label{sec_gen_strategy}
Consider a generalized Grassmannian $G/P\subset\PP(V)$ of long type (in particular, not exceptional in the sense of 
Demazure). 
Let $H_x$ be a smooth hyperplane section defined by $x\in V^\vee$, $x\neq 0$. By the Lefschetz theorem, if $\dim(G/P)>3$, we know that $H_x$ has also Picard 
number one, and therefore every automorphism is linear. We would like to be able to extend any  such  automorphism to an automorphism of $G/P$ fixing $[x]$.  Our strategy will be the following. 

\medskip\noindent Suppose first that we are in the most favourable situation, where 
$$\Hilb_{\PP^m}(G/P)=G/Q\qquad \mathrm{and}\qquad \Hilb_{\PP^{m-1}}(G/P)=G/R$$
are both $G$-homogeneous (in general, we only know that each connected component is $G$-homogeneous). 
Suppose moreover that \emph{Property (UE)} holds, in the sense that each $\PP^{m-1}$ contained in $G/P$ can be 
extended to a unique $\PP^m$ (we will soon give a slightly more general definition). Then consider the 
nested variety
\[
Z_x := \{ (P_{m-1},P_m) \in \Hilb_{\PP^{m-1}}(H_x) \times \Hilb_{\PP^m}(G/P)
             \mid P_{m-1} \subset P_m \}.
\]
By Property (UE), the projection $p$ of $Z_x$ to $\Hilb_{\PP^{m-1}}(H_x)$ is  an isomorphism. 
Moreover, the projection $q$ from $Z$ to  $\Hilb_{\PP^m}(G/P)$ is birational, since more precisely it is an isomorphism 
outside $\Hilb_{\PP^m}(H_x)$, over which the fibers are copies of $\PP^m$. Since being contained in a hyperplane gives $m+1$
conditions on a $\PP^m$, the expected codimension of $\Hilb_{\PP^m}(H_x)$ is $m+1$. By \cite[Theorem 1.1]{ein-sb}, we can deduce
the following statement.

\begin{prop}
In the previous setting, where Property (UE) does hold, suppose that $\Hilb_{\PP^m}(H_x)\subset\Hilb_{\PP^m}(G/P)$ is smooth and irreducible,
of the expected codimension. 
Suppose also that $\Hilb_{\PP^{m-1}}(H_x)$ is smooth. Then it coincides with the blowup of $\Hilb_{\PP^m}(G/P)$ along 
$\Hilb_{\PP^m}(H_x)$. 
\end{prop}

This is certainly the {\it generic} situation. In case this holds true, any automorphism $f$ of $H_x$ extends to 
an automorphism of $\Hilb_{\PP^{m-1}}(H_x)$, hence to a birational automorphism $f_*$ of $\Hilb_{\PP^m}(G/P)$. But it also 
extends to a birational automorphism of the exceptional locus $\Hilb_{\PP^m}(H_x)$. So $f_*$ extends continuously to a bijection 
of $ \Hilb_{\PP^m}(G/P)$. By Zariski's main theorem, $f_*$ is thus actually a regular
automorphism of $\Hilb_{\PP^m}(G/P)=G/Q$, hence an element of $\Aut(G/Q)$, which is essentially $G$ by Theorem \ref{thm_Demazure}. 
We thus get an extension of $f$, as desired. 
\medskip

There might be several unfortunate circumstances that will oblige us, in some cases, to be more careful. 

First, we do not 
know if $\Hilb_{\PP^{m}}(H_x)$ and $\Hilb_{\PP^{m-1}}(H_x)$ are necessarily smooth when $H_x$ is, and we prefer to avoid this 
rather delicate question. Note that in the previous argument this smoothness hypothesis doesn't play any serious role: under
the hypothesis that Property (UE) holds, an automorphism $f$ of $H_x$ always induces a birational automorphism of $\Hilb_{\PP^m}(G/P)$, 
and also an automorphism of the exceptional locus $\Hilb_{\PP^m}(H_x)$. Be it smooth or not, Zariski's main theorem applies.

Another possible complication could be that $\Hilb_{\PP^{m}}(G/P)$ has several  components. 
Clearly we just need one with the required properties to make the previous argument work. So we define 
Property (UE), in general, as follows.

\begin{definition} We say that the generalized Grassmannian $G/P$ has \emph{Property (UE)} if $\Hilb_{\PP^{m}}(G/P)$ has a component $G/Q$
such that for each $\PP^{m-1}$ in $G/P$ that can be extended to some $\PP^m$ from $G/Q$, this $\PP^m$ is unique. 
\end{definition}

We wil discuss later on in which cases this property does hold. 

\smallskip
There may also be some subtleties related to the existence of non-extendable $\PP^{m-1}$'s. 
Suppose for simplicity that 
$$\Hilb_{\PP^{m-1}}(G/P)=G/R\cup G/S$$
where we distinguish the 
linear spaces parametrized by $G/R$ that extend to one dimension bigger, from those parametrized by $G/S$ that do not extend.
In this situation, $\Hilb_{\PP^{m-1}}(H_x)$ will a priori also be disconnected, being the union of $E_{x} \subset G/R$ and 
$N_{x} \subset G/S$. An automorphism $f$ of $H_x$ being given, we would be unable to induce a birational automorphism of 
$\Hilb_{\PP^{m}}(G/P)$ if $f_*$ could send $E_x$ to $N_x$. In order to exclude this, we observe that $E_x$ always contains the 
set of $\PP^{m-1}$'s in $H_x$ that can be extended to a $\PP^{m}$ that is also contained in $H_x$. This set is a $\PP^m$-bundle
over $\Hilb_{\PP^{m}}(H_x)$, and must be preserved by $f_*$. So if $\Hilb_{\PP^{m}}(H_x)$ is not empty, we are safe. 
More precisely, what we need to check is the non-emptyness of the intersection of $\Hilb_{\PP^{m}}(H_x)$ with the component 
$G/Q$ of  $\Hilb_{\PP^{m}}(G/P)$ used in Property (UE). 

Note that by semi-continuity, it is enough to establish this non emptyness when $x$ is generic. Being contained in $H_x$, 
for the linear spaces parametrized by $G/Q$, 
is equivalent to the vanishing of a section of an irreducible homogeneous vector bundle. We therefore have several
classical tools in hand that will allow to treat this question. This will be done in Section \ref{sec_proof}. Moreover we will check in Section \ref{sec_short_root}
that the same approach works pretty well in the short root case.\smallskip

Finally, we will need a different approach when Property (UE) does not hold in $G/P$. In these few such cases, we will 
use the same circle of ideas but with quadrics instead of linear spaces. This will be discussed in Section \ref{sec_quadrics}.

\subsection{The Unique Extension property}
Let us discuss  when Property (UE) holds in the long case. 

\begin{prop} 
\label{prop_UE_property_long_case}
Let $X=G/P$ be a Grassmannian of long type, of dimension bigger than three. Then
\begin{enumerate}
\item $m>1$ if and only if $X$ is neither a Lagrangian Grassmannian $\IG(n,2n)$, nor the adjoint variety $G_2/P_1$.
\item Property (UE) holds if and only if $X$ is neither an orthogonal Grassmannian $B_n/P_k=\OG(k,2n+1)$ with $k\leq n-k$, nor $F_4/P_1$.
\end{enumerate}
\end{prop}

\begin{proof}
Both statements are simple consequences of the Dynkin diagram descriptions of the linear spaces on generalized Grassmannians. 
Indeed, $m=1$ would mean that the marked Dynkin diagram defining $X$ does not contain any marked diagram of type $A_2$. 
This is only possible if the node is extremal and bounded by a multiple edge, hence the first claim. 

For the second claim, observe that the components of $\Hilb_{\PP^{m-1}}(G/P)$ parametrizing extendable linear subspaces
are in correspondence with marked subdiagrams of type $A_{m-1}$ contained in marked subdiagrams of type $A_{m}$. Suppose that 
such a diagram $A$ of type $A_{m}$ is obtained by erasing a 
set $S$ of nodes in the Dynkin diagram $D$ of $\Lie(G)$. Since $m$ is maximal, the unmarked extremal node $e$ of $A$ must be an extremal 
node of $D$, or be connected to another node by a multiple edge; but 
the latter possibility is excluded if $X$ is neither of type $B$ nor $F_4/P_1$. 

If we suppose that $X$ is neither of type $B$ nor $F_4/P_1$, we can  therefore conclude that the corresponding diagram $A'$ of type $A_{m-1}$ is obtained by erasing $S\cup\{e\}$. Then the parameter spaces for the corresponding families of $\PP^m$'s and $\PP^{m-1}$'s are $G/P_S$ and $G/P_{S\cup \{e\}}$, and the canonical projection morphism $G/P_{S\cup \{e\}}\longrightarrow G/P_S$ defines the 
canonical extension map we need for Property (UE) to hold. Indeed, the variety parametrizing pairs $(\PP^{m-1}\subset \PP^m)$ 
in $G/P$ is a $\PP^m$-bundle over  $G/P_S$ as well as $\Hilb_{\PP^{m-1}}(G/P)=G/P_{S\cup \{e\}}$, and the canonical extension 
map provides an inverse to the obvious projection from the former to the latter. 

Let us now deal with the case when $X$ is of type $B$. Note that in an orthogonal Grassmannian $\OG(k,2n+1)$ of type $B$, which is long for $k<n$, 
projective spaces are either of the 
form $\{V_{l}\subset U\subset V_{k+1}\}$, with $V_{k+1}$ isotropic, or of the form $\{V_{k-1}\subset U\subset V_{l}\}$ with $V_{l}$ isotropic of dimension $l>k$. This implies that $m=\max(k,n-k)$. For $k\le n-k$, we get $\PP^{m-1}$'s in 
$\OG(k,2n+1)$ parametrized by the isotropic flag manifold
$\OFl(k-1,n-1,2n+1)$. Moreover, given such a $\PP^{m-1}$ defined 
by a pair $V_{k-1}\subset V_{n-1}$, it is contained in a $\PP^m$ 
defined by a pair $U_{k-1}\subset U_{n}$ if and only if $U_{k-1}=V_{k-1}$ and $U_{n}\supset V_{n-1}$. Such isotropic 
spaces of dimension $n$ are parametrized by the associated conic
in $\PP(V_{n-1}^\perp/V_{n-1})$, showing that Property (UE) does not hold. For $k>n-k$, the extendable $\PP^{m-1}$'s
are defined by a pair $V_{1}\subset V_{k+1}$, and they are uniquely extendable; so Property (UE) does hold in this range.

For $F_4/P_1$, we know that $\Hilb_{\PP^1}(F_4/P_1)=F_4/P_2$ and $\Hilb_{\PP^2}(F_4/P_1)=F_4/P_3$ both have dimension $20$. The incidence variety of pairs of planes and lines in them is thus a $\PP^2$-fibration above both of them. In particular Property (UE) does not hold.
\end{proof}

\subsection{Short roots} 
\label{sec_short_root}
For the short cases, that is in types CFG (recall that type B reduces to type D), our strategy does apply 
in many cases. In fact, although the space of lines in a generalized Grassmannian of short type is not homogeneous,
it appears that its varieties of linear spaces of maximal and almost maximal dimensions are homogeneous in most cases. 
We will leave aside the case of $G_2$ since $G_2/P_2$ is just a five dimensional quadric. 

\subsubsection{Type C} \label{sec_type_C}
For $G$ of type $C$ and $P$ a maximal parabolic subgroup corresponding to a short simple root, $G/P$ is 
an isotropic Grassmannian $\IG(k,2n)$, with $2\le k<n$ (we exclude $k=1$, which is exceptional in the sense of 
Demazure since we just get a projective space). We discuss those explicitly. 

\begin{lemma}
\label{lemma_generic_Sp_grassm_UE}
The maximal dimension of a linear space in  $G/P=\IG(k,2n)$, with $k<n$, is $m=\max(k,2n-2k+1)$. 
Linear spaces of maximal dimension are parametrized by $\Hilb_{\PP^m}(G/P)$ which is $\IG(k-1,2n)$
when $k\le 2n-2k$, $\IG(k+1,2n)$ when $k>2n-2k+1$, and their union when  $k=2n-2k+1$.

Moreover Property (UE) does hold. 
\end{lemma}

\proof Recall that a linear space in an ordinary Grassmannian $G(k,2n)$ is necessarily obtained as 
the set of spaces $U\subset\CC^n$ such that $V_i\subset U\subset V_j$, where $V_i\subset V_j$ are 
fixed and either $j=k+1$ or $i=k-1$. In the former case, we get a projective space contained in $\IG(k,2n)$ 
when  $V_j=V_{k+1}$  is isotropic, and its dimension is $k-i\le k$.  In the latter case, we need $V_i=V_{k-1}$ to
be isotropic and $V_j$ to be contained in $V_i^\perp$, and the dimension is $j-k\le 2n-2k+1$. The first claim
easily follows. 

In order to prove that Property (UE) holds, just observe that extendable $\PP^{m-1}$ are either of the form 
$V_1\subset U\subset V_{k+1}$ or $V_{k-1}\subset U\subset V_{2n-k}
\subset V_{k-1}^\perp $, and they clearly are uniquely extendable.
\qed 

\medskip\noindent As a consequence, in all cases $2<k<n$ our general strategy will work. Beware that for $k=2$ 
there is a problem since $\IG(1,2n)=\PP^{2n-1}$ has bigger automorphism group than the symplectic group! This 
case will be considered in greater detail in section \ref{sec_type_Jordan}. 

\subsubsection{Remaining cases}
At this point we have checked that our general strategy works fine except for the following generalized Grassmannians:
\begin{itemize}
\item in the long case, $\OG(k,2n+1)$ for $k\leq n-k$, $\IG(n,2n)$ and $F_4/P_1$;
\item in the short case, $\IG(2,2n)$ and $F_4/P_3$, $F_4/P_4$. 
\end{itemize}
The cases of $\IG(n,2n)$ and $F_4/P_1$  will be dealt with in the next section by replacing 
linear spaces with quadrics. The cases of $\IG(2,2n)$, $F_4/P_3$ and $F_4/P_4$ will be discussed later 
and turn out to display some very special features. 

\section{Playing with quadrics} 
\label{sec_quadrics}
Let $G/P$ be a generalized Grassmannian, and let $l$ be the maximal dimension of quadrics that are contained inside $G/P$.

\begin{definition} We say that the generalized Grassmannian $G/P$ has \emph{Property (UE) for quadrics} if $\Hilb_{\QQ^l}(G/P)$ has a component $G/Q$
such that for each $\QQ^{l-1}$ (possibly singular) in $G/P$ that can be extended to some $\QQ^l$ from $G/Q$, this $\QQ^l$ is unique. 
\end{definition}



\subsection{Type B: the orthogonal Grassmannians $\OG(k,2n+1)$, $k\leq n-k$} 
\label{sec_ort_grass_quadrics}
Recall that quadrics in a Grassmannian
are linear sections of copies of $G(2,4)$ (which we will refer to as type $0)$ quadrics), unless their linear span is contained in the Grassmannian. The latter kind of quadrics can be subdivided into two classes:
\begin{itemize}
\item[$I)$] Quadrics of subspaces $V_k\subset V_{2n+1}$ containing a fixed $V_{k-1}$ and contained inside a fixed $V_j$. 
For this quadric to be non-empty inside $\OG(k,2n+1)$, $V_{k-1}$ must be isotropic and $V_j\subset V_{k-1}^\perp$. Then the quadric $\QQ^{j-k-1}\subset \PP(V_j/V_{k-1})$ defined by restricting the quadratic form of $\SO_{2n+1}$ to $\PP(V_j/V_{k-1})$ is contained inside $\OG(k,2n+1)$.
\item[$II)$] Quadrics of subspaces $V_k\subset V_{2n+1}$ containing a fixed $V_{i}$ and contained inside a fixed $V_{k+1}$. These quadrics, when they exist, are thus contained inside $\PP(V_{k+1}/V_i)$ and their dimension is equal to $k-i-1$.
\end{itemize}
The maximal dimension of quadrics of type $II$ is $k-1$, while the maximal dimension of quadrics of type $I$ is $2n-2k+1$. As a consequence, under the hypothesis $k\leq n-k$, maximal and submaximal quadrics inside $\OG(k,2n+1)$ are of type $I$ (note that linear sections of copies of $G(2,4)$ are never maximal or submaximal unless $k=1$ or $n=4$ and $k=2$, in which case they are submaximal but not extendable). Moreover, submaximal quadrics of type $I$ (parametrized by a flag $V_{k-1}\subset V_{2n-k+1}$) are uniquely extendable (to the maximal quadric defined by the flag $V_{k-1}\subset V_{k-1}^\perp$). Thus Property (UE) holds for quadrics in $\OG(k,2n+1)$ for $k\leq n-k$ and the general strategy applies for these varieties.

\begin{remark}
Note that maximal quadrics inside $\OG(k,2n+1)$ when $k\leq n-k$ are parametrized by the homogeneous variety $\OG(k-1,2n+1)$, as expected from the theory of Tits shadows. Indeed, erasing the $k-1$-th simple root from the marked Dynkin diagram of $\OG(k,2n+1)$, one obtains the marked Dynkin diagram of $\QQ^{2n-2k+1}$. Thus, the projection from the incidence flag variety $\OFl(k-1,k,2n+1)$ to $\OG(k-1,2n+1)$ has fibers isomorphic to $\QQ^{2n-2k+1}$, showing that $\OG(k-1,2n+1)$ is the parameter space of a family of quadrics of dimension $2n-2k+1$ contained inside $\OG(k,2n+1)$.
\end{remark}

\subsection{Type C: the Lagrangian Grassmannian $\IG(n,2n)$} 
Inside  $\IG(n,2n)$ there are no quadrics whose linear span is contained in the Grassmannian. Therefore a quadric in $\IG(n,2n)$ must be a linear section of a set of
$n$-planes $U$ such that $V_{n-2}\subset U\subset V_{n+2}$ for some fixed $V_{n-2}\subset V_{n+2}$.
For this set to meet $\IG(n,2n)$ non trivially, we need $V_{n-2}$ to be isotropic. Since each $U$ in the 
intersection will be contained in $V_{n-2}^\perp$, we can replace $V_{n+2}$ by $V_{n+2}\cap V_{n-2}^\perp$;
but if this intersection is a proper subspace  $V_{n+2}$ our intersection is in fact linear. In other 
words we need $V_{n+2}=V_{n-2}^\perp$ to get a quadric, and this quadric will be a linear section of 
a copy of $G(2,4)$, i.e. $\IG(2,4)\cong\QQ^3$. This implies:

\begin{lemma}
\label{lem_isotr_grass}
The maximal dimension of a quadric in  $G/P=\IG(n,2n)$ is $l=3$. Moreover,
$\Hilb_{\QQ^3}(G/P)=\IG(n-2,2n)$, while  $\Hilb_{\QQ^2}(G/P)$ is a $\PP^3$-bundle over $\IG(n-2,2n)$.
\end{lemma}

 In particular Property (UE) works fine for quadrics, and our 
general strategy applies verbatim when we replace linear spaces by quadrics.

\subsection{Type F: the adjoint variety $F_4/P_1$} \label{sec_type_F}
The Tits theory of shadows implies that there exists a natural morphism 
$$F_4/P_4\longrightarrow \Hilb_{\QQ^5}(F_4/P_1),$$
 as can be read on the following 
marked diagram:

\begin{center}
\setlength{\unitlength}{4mm}
\thicklines
\begin{picture}(10,2)(.5,1.3)
\multiput(2,2)(2,0){3}{$\circ$}
\put(0,2){$\bullet$}
\multiput(0.3,2.2)(4,0){2}{\line(1,0){1.8}}
\multiput(2.3,2.1)(0,.2){2}{\line(1,0){1.8}}
\put(2.8,1.95){$>$}

\put(-.3,1.7){\line(0,1){1}}
\put(-.3,1.7){\line(1,0){5}}
\put(-.3,2.7){\line(1,0){5}}
\put(4.7,1.7){\line(0,1){1}}

\color{blue}
\put(6,2){$\bullet$}
\end{picture}
\end{center}
We need a more precise statement.

\begin{prop} \label{quadrics-F4P1}
The maximal quadrics in $F_4/P_1$ have dimension $5$, and 
$$\Hilb_{\QQ^5}(F_4/P_1)\simeq F_4/P_4.$$
\end{prop}

\begin{proof}
In order to prove those statements, we claim it is enough to check that a plane in $F_4/P_1$ is contained 
only in a one-dimensional, irreducible family of five dimensional quadrics. Indeed, this will imply that 
the incidence variety parametrizing pairs $(\PP^2\subset\QQ^5)$ in $F_4/P_1$ is irreducible of dimension $21$ 
(recall that planes in $F_4/P_1$ are parametrized by $F_4/P_3$, whose dimension is $20$). Since the projection 
map to $\Hilb_{\QQ^5}(F_4/P_1)$ has generic fibers isomorphic to $\OG(3,7)\simeq \QQ^6$, we will deduce that 
this space is irreducible of dimension $15$. In particular it must be $F_4$-homogeneous and the morphism
given by Tits shadows must be an isomorphism. Note that this will also imply that there is no six-dimensional 
quadric $Q$ in $F_4/P_1$, since any plane in this quadric  is contained in a family of hyperplane 
sections of $Q$ of dimension bigger than one. (Alternatively, hyperplane sections of six dimensional
quadrics could be smooth or singular, contradicting the homogeneity of the space of five dimensional quadrics.)

In order to have a more concrete grasp on $F_4/P_1$, we shall use the classical $\ZZ_2$-grading 
$\ff_4=\fso_9\oplus\Delta$. Here $\Delta$  is the $16$-dimensional spin representation, and the 
closed orbit in $\PP\Delta$ is the orthogonal Grassmannian $\OG(4,9)$, a spinor variety. One way to characterize the 
adjoint variety in $\PP(\ff_4)$ is through its equations. Recall that the ideal is generated by quadrics, that we are going to write down explicitly. 

For this, observe that $S^2\fso_9=S^2(\wedge^2V_9)$ maps to $\wedge^4V_9$ which is an irreducible $\fso_9$-module, 
obtained as the Cartan square of the spin module $\Delta$. In other words, there exists a surjective equivariant map
$j : S^2\Delta\rightarrow \wedge^4V_9$. Obviously this map must be compatible with the two natural embeddings of 
the variety of maximal isotropic spaces:  for any pure 
spinor $\delta$, representing an isotropic four-plane $F$, $j(\delta^2)$ has to represent $F$ in the usual 
Pl\"ucker embedding, which means that $j(\delta^2)$ must be a generator of $\wedge^4F$. 

\begin{lemma}\label{f4-so9} 
Consider $X\in \fso_9$ and a nonzero $\delta\in\Delta$. Then $X+\delta$ belongs to the 
cone over the adjoint variety $F_4/P_1$ if and only if 
\begin{itemize}
\item $\delta$ is a pure spinor, associated to some maximal isotropic subspace $F$ of $V_9$,
\item $X$ belongs to $\wedge^2F\subset \wedge^2V_9\simeq \fso_9$,
\item $X\wedge X=j(\delta^2)$ in $\wedge^4F\subset \wedge^4V_9$.
\end{itemize}
\end{lemma} 

For a fixed nonzero $\delta$, the last equation defines a smooth quadric inside the projective space $\PP(\wedge^2F\oplus\CC\delta)$. 
So we get an open subset of $F_4/P_1$ (whose complement is $\OG(2,9)$)
as  a bundle over the spinor variety $\OG(4,9)$, with fiber $\QQ^5$ minus a smooth hyperplane section.  

\begin{proof} We denote by $V_\omega$ the irreducible $\ff_4$-module of highest weight $\omega$, and by 
$W_\theta$ the irreducible $\fso_9$-module of highest weight $\theta$. According to LiE, 
$$S^2\ff_4=V_{2\omega_1}\oplus V_{2\omega_4}\oplus \CC, \qquad S^2V_{\omega_4}=V_{2\omega_4}\oplus V_{\omega_4}\oplus \CC. $$
From the first decomposition, we get that the quadratic equations of the adjoint variety are given by $V_{2\omega_4}\oplus \CC$. The second decomposition gives its restriction to $\fso_9$, starting from the restriction of $V_{\omega_4}=
W_{\omega_1}\oplus W_{\omega_4}\oplus \CC$. Indeed we deduce that $V_{2\omega_4}\oplus\CC\simeq 
S^2(W_{\omega_1}\oplus W_{\omega_4})$.
Hence 
$$I_2(F_4/P_1)=(W_{2\omega_1}\oplus \CC)\oplus (W_{\omega_1+\omega_4}\oplus W_{\omega_4})\oplus (W_{2\omega_4}\oplus W_{\omega_1}\oplus \CC).$$ 
Now we compare with the decomposition of $S^2\ff_4=S^2(\fso_9\oplus W_{\omega_4})$. We have 
$$S^2\fso_9=\mathbf{W_{2\omega_2}}\oplus W_{2\omega_1}\oplus W_{2\omega_4}\oplus \CC, $$
$$\fso_9\otimes  W_{\omega_4}=\mathbf{W_{\omega_2+\omega_4}}\oplus W_{\omega_1+\omega_4}\oplus W_{\omega_4},$$
$$S^2W_{\omega_4}=W_{2\omega_4}\oplus W_{\omega_1}\oplus \CC.$$
We can see that all the terms above that are not in bold do appear in the equations, and that there is only one 
ambiguity for $W_{2\omega_4}$, which does appear twice, but only once in the equations. 

Let us analyze what these equations mean for $X+\delta\in W_{\omega_2}\oplus W_{\omega_4}$. Note first that 
since $\delta$ has to verify the equations coming from $W_{\omega_1}\oplus \CC\subset S^2W_{\omega_4}$, it has to be a pure spinor. 
Then the mixed equations from $W_{\omega_1+\omega_4}, W_{\omega_4}\subset W_{\omega_2}\otimes W_{\omega_4}$
mean that the image of $X\wedge\delta\in \wedge^2V_9\otimes \Delta$ by contraction to $V_9\otimes\Delta\cong W_{\omega_1+\omega_4}\oplus W_{\omega_4}$ is zero. 
By homogeneity, if $\delta$ is a pure spinor we can let $\delta=1=\delta_E$, corresponding to our preferred isotropic
space $E$; it is then a straightforward computation that the resulting condition is that $X$ must belong to $\wedge^2E$.
Finally, we need to take into account the equations from $W_{2\omega_4}$, which appears both 
in $S^2\fso_9$ and $S^2W_{\omega_4}$.  We can write down these equations in the form
$a(X\wedge X)=bj(\delta^2),$ for some scalars $(a,b)\ne (0,0)$. Obviously $a\ne 0$, since otherwise $\delta$
would be forced to vanish. Also $b\ne 0$, since otherwise $X$ would belong to a cone over a Grassmannian and 
we would get singularities. So we can normalize $j$ (or $X$) so that our equations have the required form. 
\end{proof}

\smallskip
Now, let us choose a maximal  isotropic space $E$ in $V_9$ and a line 
$L\subset E$. Then $\Pi=\PP(L\wedge E)$ is a projective plane in the adjoint variety of $\fso_9$, 
hence also in that of $\ff_4$. Moreover there is a unique five-dimensional quadric $Q_\Pi$ in the 
adjoint variety of $\fso_9$ that contains $\Pi$, namely the 
variety parametrizing isotropic planes that contain $L$. Let us look for the other quadrics $Q$ containing $\Pi$. 
Obviously, a three-dimensional space in the span of $Q$, containing $\Pi$, has to cut $F_4/P_1$ along the union
of two planes. 

\begin{lemma} Consider a point $[X+\delta_F]$ of $F_4/P_1$, where $\delta_F\in\Delta$ is a pure spinor defining a maximal 
isotropic subspace $F$ of $V_9$, and $X\in \wedge^2F$ with $X\wedge X=j(\delta^2)$.  
Then $\langle \Pi, X+\delta_F\rangle $ meets  $F_4/P_1$ along the
union of $\Pi$ and another plane if and only if there exists a three-dimensional subspace $M$ of $V_9$, with 
$L\subset M\subset E\cap F$, such that $X\in \wedge^2M+L\wedge F$. 
\end{lemma} 

\begin{proof} 
Consider $e_0\wedge e+X+\delta_F$ for $e_0$ a generator of $L$ and $e$ a vector in $E$.
By  Lemma \ref{f4-so9} it belongs to the adjoint variety if and only if $e_0\wedge e+X\in\wedge^2F$, which is 
equivalent to $e_0\wedge e\in\wedge^2F$, and $X\wedge e_0\wedge e=0$. We want this condition to be verified on a 
hyperplane in $\Pi$, and any such hyperplane is of the form $\PP(L\wedge M)$ for $L\subset M\subset E\cap F$. 
The claim follows.
\end{proof} 

\smallskip
So let us consider a deformation $\Lambda$ of the linear span of $Q_\Pi$, namely $L\wedge L^\perp$, made of vectors 
verifying the previous property. Choose a complement $N=\langle e_2,e_3,e_4\rangle$ to $L=\langle e_1\rangle$ in $E$, 
and observe that $\Lambda$ must be contained 
in $\wedge^2N\oplus L\wedge L^\perp\oplus\Delta$. Suppose that the projection of $\Lambda$ to $L\wedge L^\perp$ is an
isomorphism, meaning that every vector in $\Lambda$ can uniquely be written as 
$\omega(v)+e_1\wedge v+\delta(v)$ for some $v\in L^\perp/L$. When $\delta(v)\ne 0$, it has to be a pure spinor 
defined by the four space $F=\langle p,q,e_1,v\rangle $ with $\omega(v)=p\wedge q$.  Recall that   a
pure spinor $\delta_F$ associated to $F$ is characterized (up to scalar) by the property that $f.\delta_F=0$ 
for any $f\in F$. Using this, we compute that if $v=ae_2+be_3+ce_4+ze_1+a'f_2+b'f_3+c'f_4$, then 
$$\delta(v)=\Delta (a'e_{134}+b'e_{142}+c'e_{123}+ze_{1234}), \qquad \omega(v)=\Omega(a'e_{34}+b'e_{42}+c'e_{23})$$
for some constants $\Delta$ and $\Omega$. Finally, $\PP(\Lambda)$ intersects $F_4/P_1$ along a quadric exactly 
when $j(\delta(v)^2)$ can equal $e_1\wedge v\wedge \omega(v)$ up to some prescribed constant, with yields a condition 
$\Omega =t\Delta^2$ for some constant $t\ne 0$. We thus get a one-dimensional family of quadrics, parametrized by $\Delta\in\CC$. 

What are the missing quadrics? Their linear span must fail the transversality property we started from, that is,
they must meet $\wedge^2N\oplus\Delta$ non trivially. Since this must be true for any $N\subset E$, we conclude 
that in fact $\Lambda$ must meet $\Delta$ itself non trivially, and then that the only possibility is that 
$\Lambda =\langle \wedge^2E,\delta_E\rangle$. This concludes the proof. 
\end{proof}

Note that as a consequence, the incidence variety parametrizing pairs $(\PP^2\subset\QQ^5)$ in $F_4/P_1$ is 
actually $F_4/P_{3,4}$. 

\begin{prop}
\label{prop_UE_quadrics_F4/P1}
Property (UE) holds for four dimensional quadrics in $F_4/P_1$. 
\end{prop}

\begin{proof}
It is enough to check that two distinct five dimensional quadrics $Q$ and $Q'$ in $F_4/P_1$ meet in codimension 
bigger that one. But the maximal dimension for this intersection must be obtained, by semi-continuity, 
when $Q$ and $Q'$ are as "close" as possible one from the other (in the sense that the line that joins them 
belongs to the minimal orbit in the space of lines), which means that the two corresponding 
points $q$ and $q'$ in $F_4/P_4$ are joined by a line, and even a special line. Since the space of special 
lines in $F_4/P_4$ is in fact parametrized by $F_4/P_3=\Hilb_{\PP^2}(F_4/P_1)$, this means that $Q$ and $Q'$ are two 
of the quadrics we described in the proof of the previous Lemma. We may for example suppose that their linear spans 
are $L\wedge L^\perp$ and $\langle \wedge^2E,\delta_E\rangle$, and we conclude that in fact $Q\cap Q'=\Pi$
is just a projective plane. 
\end{proof}

We can therefore conclude that our 
general strategy also applies to $F_4/P_1$  when we replace linear spaces by quadrics.

\begin{remark} 
Recall that $F_4$ can be constructed by folding $E_6$. From this perspective, Proposition \ref{quadrics-F4P1} 
is a folded version of the similar statement for the adjoint variety $E_6/P_2$ of $\fe_6$, which should be 
$$ \Hilb_{\QQ^6}(E_6/P_2)\simeq E_6/P_{1,6}.$$
A map from the right hand side to the left hand side is provided by Tits shadows, as can be read on the following 
marked diagram:

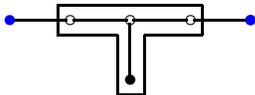
\begin{center}
\setlength{\unitlength}{4mm}
\thicklines
\begin{picture}(10,3.5)(.5,-.4)
\multiput(2,2)(2,0){4}{$\circ$}
\put(4,0){$\bullet$}
\multiput(0.3,2.2)(2,0){4}{\line(1,0){1.8}}
\put(4.2,0.3){\line(0,1){1.8}}

\put(1.8,1.7){\line(0,1){1}}\put(6.6,1.7){\line(0,1){1}}
\put(1.8,2.7){\line(1,0){4.8}}
\put(1.8,1.7){\line(1,0){2}}
\put(3.8,-.3){\line(0,1){2}}
\put(3.8,-.3){\line(1,0){.9}}
\put(4.7,-.3){\line(0,1){2}}
\put(4.7,1.7){\line(1,0){1.9}}

\color{blue}
\put(0,2){$\bullet$}\put(8,2){$\bullet$}
\end{picture}
\end{center}
\end{remark}

\section{More on type F} 
\label{sec_more_on_type_F}

In this section we discuss in more details the two most challenging generalized Grassmannians, which from 
our perspective are $F_4/P_4$ and $F_4/P_3$. 

\subsection{$F_4/P_4$, the hyperplane section of the Cayley plane} 
Recall that $F_4/P_4$  is a general hyperplane section of $E_6/P_1$, also known as the Cayley plane. 
Considering Tits shadows, we see that there exists a family of (special) planes inside $F_4/P_4$ parametrized 
by $F_4/P_2$, but there are more. 

\begin{prop} The space of maximal linear spaces inside $F_4/P_4$ is 
$$\Hilb_{\PP^5}(F_4/P_4)=F_4/P_1.$$
\end{prop} 

Note that this statement is definitely not a consequence of the theory of Tits shadows. 
But this is another instance of a phenomenon that we already observed in generalized Grassmannians of 
short types: although the lines they contain are not parametrized by homogeneous spaces, it happens to be the case for
their maximal linear spaces. 

\begin{proof}
Recall that maximal linear spaces in $E_6/P_1$ are $\PP^5$'s, parametrized 
by the adjoint variety $E_6/P_2$. Those that are contained in $F_4/P_4$ are parametrized by a codimension six subvariety 
of $E_6/P_2$. Since the latter has dimension $21$, we must get a $F_4$-variety of dimension $15$, hence one of the two 
generalized Grassmannians of this dimension, because this is the minimal possible dimension for a projective variety with a non trivial $F_4$-action. Among those two, namely $F_4/P_1$ and $F_4/P_4$, only the first one is 
embedded in $E_6/P_2$. Indeed, $\fe_6$ decomposes as $\ff_4\oplus V_{\omega_4}$ as a $\ff_4$-module, where 
$V_{\omega_4}$ is the $26$-dimensional representation. So there is only one way to embed $F_4/P_4$ inside $\PP(\fe_6)$,
as the closed orbit inside $\PP (V_{\omega_4})$. 

Let us check that $E_6/P_2$ does not contain this closed orbit by looking at the quadratic equations of the adjoint variety.
The discussion will be strikingly similar to that of the proof of Lemma \ref{f4-so9}.
According to LiE \cite{LiE}, we have 
$$S^2\fe_6= U_{2\omega_2}\oplus U_{\omega_1+\omega_6}\oplus\CC, \qquad U_{\omega_1}\otimes U_{\omega_6}=
U_{\omega_1+\omega_6}\oplus \fe_6\oplus\CC,$$
if we denote by $U_{\omega}$ the irreducible $\fe_6$-module of highest weight $\omega$. The first decomposition shows that 
the equations of the adjoint variety are given by $U_{\omega_1+\omega_6}\oplus\CC$, while the second decomposition 
allows to compute as a $\ff_4$-module, since $U_{\omega_1}$ and $U_{\omega_6}$ have the same restrictions 
$V_{\omega_4}\oplus\CC$. We get 
$$I_2(E_6/P_2)=V_{2\omega_4}\oplus V_{\omega_3}\oplus 2V_{\omega_4}\oplus 2\CC.$$
Now we compare with the decomposition of $S^2\fe_6=S^2(\ff_4\oplus V_{\omega_4})$. We have 
$$S^2\ff_4=\mathbf{V_{2\omega_1}}\oplus V_{2\omega_4}\oplus \CC, $$
$$\ff_4\otimes  V_{\omega_4}=\mathbf{V_{\omega_1+\omega_4}}\oplus V_{\omega_3}\oplus V_{\omega_4},$$
$$S^2V_{\omega_4}=V_{2\omega_4}\oplus V_{\omega_2}\oplus \CC.$$
We can see that all the terms above that are not in bold do appear in the equations, and that there is only one 
ambiguity for $V_{2\omega_4}$, which does appear twice, but only once in the equations. 

So we discuss what it means for $X+v\in \ff_4\oplus V_{\omega_4}$ to verify these equations. First observe 
that $v$ has to obey the equations from $V_{\omega_2}\oplus \CC$, which forces it to belong to the cone 
over $F_4/P_4$. Moreover the mixed equations defined by $V_{2\omega_4}$ are of the form $aq(X)=bv^2$, where 
$(a,b)\ne 0$ and $q:S^2\ff_4\rightarrow V_{2\omega_4}$ is a projection map. Here again $b\ne 0$, since otherwise 
for $v$ fixed the equations in $X$ would be homogeneous, and we would get singularities. But then, if $X=0$, 
necessarily $v=0$. This proves that $E_6/P_2\cap\PP(V_{\omega_4})=\emptyset$, which is what we wanted to prove. 
\end{proof}

Note that we get a similar conclusion to what we obtained for the adjoint variety of $\ff_4$, namely: an open subset 
of the adjoint variety of $\fe_6$ (whose complement is precisely the adjoint variety of $\ff_4$) can be described
as a bundle over $F_4/P_4$, with fiber $\QQ^6$ minus a smooth hyperplane section. 

\medskip
Now consider the extendable $\PP^4$'s in $E_6/P_1$, which are parametrized by $E_6/P_{2,6}$. In particular Property (UE) does hold,
which means they are all extendable in a unique way. Those of $F_4/P_4$ are parametrized by a subvariety of codimension $5$ (the zero locus of a general section of a rank 5 bundle), hence of dimension $21$. 
By dimension count, those that are extendable inside $F_4/P_4$ form an irreducible subvariety of dimension $20$. 

When we consider a hyperplane section  $H_x$, the natural map $$\Hilb_{\PP^5}(F_4/P_4)=F_4/P_1\dashrightarrow \Hilb_{\PP^4}(H_x)$$
is therefore not dominant, but only birational to a hypersurface.
\smallskip

So let us consider quadrics. Looking at Tits shadows, we can see that there is a family of maximal quadrics on 
the Cayley plane $E_6/P_1$, parametrized by the dual Cayley plane $E_6/P_6$, as can be read on the following marked 
diagram: 

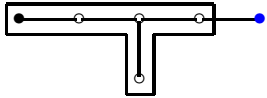
\begin{center}
\setlength{\unitlength}{4mm}
\thicklines
\begin{picture}(10,3.5)(.5,-.4)
\multiput(2,2)(2,0){3}{$\circ$}
\put(4,0){$\circ$} \put(0,2){$\bullet$}
\multiput(0.3,2.2)(2,0){4}{\line(1,0){1.8}}
\put(4.2,0.3){\line(0,1){1.8}}

\put(-.2,1.7){\line(0,1){1}}\put(6.7,1.7){\line(0,1){1}}
\put(-.2,2.7){\line(1,0){6.9}}
\put(-.2,1.7){\line(1,0){4}}
\put(3.8,-.3){\line(0,1){2}}
\put(3.8,-.3){\line(1,0){.9}}
\put(4.7,-.3){\line(0,1){2}}
\put(4.7,1.7){\line(1,0){2}}

\color{blue}
\put(8,2){$\bullet$}
\end{picture}
\end{center}

In fact these quadrics are exactly the $\OO$-lines in the octonionic geometry of the Cayley plane, 
considered as the projective plane $\OO\PP^2$ over the octonions. 
Let us prove that they are the maximal quadrics.

\begin{prop}\label{hilbQ-E6P1}  $\Hilb_{\QQ^8}(E_6/P_1)=E_6/P_6.$ \end{prop}

Any $\QQ^8$ contains a $\PP^4$, and we know that $E_6$ acts almost transitively on the space of  $\PP^4$'s 
in the Cayley plane. More precisely, there are two families of such $\PP^4$'s: the extendable ones are parametrized by 
$E_6/P_{2,6}$, of dimension $26$, while the non extendable ones are parametrized by $E_6/P_5$, of dimension $25$.
Looking at the incidence variety parametrizing pairs $(\PP^4\subset\QQ^8)$ in $E_6/P_1$, we see that it is enough to 
prove the following result:

\begin{prop}\label{ext-next} 
An extendable $\PP^4$ in $E_6/P_1$ is contained in a unique quadric $\QQ^8$, 
while a non extendable  $\PP^4$ is  contained in a one-dimensional family of such quadrics. 
\end{prop}

Before starting the proof we need a few reminders. 

Classically, the Cayley plane can be described as 
the closure in $\PP(\cJ_3(\OO_\CC))$ of the space of Hermitian matrices of the form
$$M_{x,y}=\begin{pmatrix} 1&x&y \\ \bar{x}&\bar{x}x&\bar{x}y \\ \bar{y}&\bar{y}x&\bar{y}y \end{pmatrix}, \qquad x,y\in\OO. $$
Letting $y=0$ we get one of the maximal quadrics in the Cayley plane, that we denote by $Q_0$. 
Being of even dimension, this quadric
contains two distinct families of maximal linear spaces,  representatives of which  can be chosen as follows. 

Recall first that the octonions allow to give a concrete geometric interpretation of triality for $Spin_8$,
that we see as an identification between the three homogeneous spaces associated to the three extreme nodes of 
the Dynkin diagram of type $D_4$: classically, these homogeneous spaces are the quadric $\QQ^6$, and the two
spinor varieties $\OG(4,8)^+$ and $\OG(4,8)^-$. The point is that, for any non zero $z\in \OO$ such that 
$|z|^2=z\bar{z}=0$, the vector subspaces $L_z:=z\OO$ and $R_z=\OO z$ are isotropic (since the octonionic norm
is multiplicative) of maximal dimension, yielding explicit isomorphisms between $\QQ^6$ and the two spinor varieties 
$\OG(4,8)^\pm$. Note moreover that $L_z$ (resp. $R_z$) can also be defined as the kernel of the left (resp. right) 
multiplication by $\bar{z}$. Moreover, one can check that 
$$\dim (L_z\cap R_y)=3 \iff y\in R_{\bar z} \iff z\in L_{\bar y} \iff yz=0.$$
Consider the corresponding $\PP^4$'s in $Q_0$, namely
$$A_z:=\langle M_{x,0}, \; x\in L_z\rangle \qquad \mathrm{and} \qquad B_z:=\langle M_{x,0}, \; x\in R_z\rangle.$$
Although these two linear spaces could look completely indistinguishable, we claim they are not: indeed, 
$A_z$ is extendable, while $B_z$ is not. More precisely, $A_z$ can be extended to a unique $\PP^5$ by taking its 
linear span with $M_{0,z}$, since $\bar{x}z=0$ for any $x\in L_z$. The fact that $B_z$ cannot be extended is 
a straightforward computation, or also follows from the considerations below, showing that $A_z$ and $B_z$ 
have distinct properties with respect to quadrics, and must therefore belong to different families. 

We are now ready to attack the proof of the Proposition. 

\begin{proof} We shall prove that $A_z$ and $B_z$ have the announced properties. 

\smallskip\noindent {\bf Extendable case}. 
Let us prove that $A_z$ is contained in no other quadric than $Q_0$. In order to check this,
consider a $\PP^5$ containing $A_z$ and suppose that its intersection with $\OO\PP^2$ is a quadric.
If we can prove that this quadric is necessarily contained in $Q_0$, our claim will follow. 

So suppose that our $\PP^5$ is generated by $A_z$ and a Hermitian matrix of the form 
$$N=\begin{pmatrix} 0&a&b \\ \bar{a}&v&c \\ \bar{b}&\bar{c}&w \end{pmatrix}. $$
Then $M_{u,0}+sN$ has rank one if and only if $s=0$ or 
$$v=\langle a, u\rangle +s a\bar{a}, \quad w=sb\bar{b}, \quad c=(\bar{u}+s\bar{a})b,$$
and we want these conditions to define a unique hyperplane.  

If $b=0$, we immediately get the conditions $c=w=0$, and our linear space must be contained in $Q_0$. If $c=0$ then $v=0$ and one of the minors implies that $w(sa+u)=0$, thus $w=0$ and we come back again to $Q_0$. 
So suppose that $b\ne 0$ and $c\neq 0$. 

Note that the last equation is octonionic. 
Since $c\ne 0$, there should therefore exist a scalar $\gamma$ such that $\bar{a}b=\gamma c$, and a linear form 
$\phi$ on $L_z$ such that $\bar{u}b=\phi(u)c$ for all $u\in L_z$, so that we are reduced to the 
scalar equation $1=\phi(u)+s\gamma$. Note that on the kernel of $\phi$ we get the identity $\bar{u}b=0$, 
which requires that $b$ is isotropic and $u$ belongs to $L_b$. But if $b\ne 0$, $L_b$ being in the same family of isotropic
spaces as $L_z$ needs to meet it in even dimensions, so certainly not along a hyperplane. This means that in fact 
$\phi=0$ and $b$ is proportional to $z$. 

If $w\ne 0$, we are left with the hyperplane at infinity, 
 namely the space of matrices of the form
$$\begin{pmatrix} 0&a&sb \\ \bar{a}&sv&sc \\ s\bar{b}&s\bar{c}&sw \end{pmatrix}, \qquad a\in L_z, s\in\CC. $$
For such a matrix to have rank one, all the two by two minors must vanish, and we deduce that $c=w=0$,
and then $bv=0$. So $v=0$, which means we are inside the unique $\PP^5$ that extends $A_z$.

If the condition $w=0$ is trivial,  we remain with two linear conditions 
that must be proportional, which yields the identities $a\bar{a}=v\gamma$ and $\langle a, u\rangle =0$ for all $u\in L_z$. 
In particular $a$ belongs to  $L_z^\perp=L_z$, so it must be isotropic, $\bar{a}b=0$ and 
we are back to the hyperplane at infinity.

In any case, we conclude that the $\PP^5$ we started with has to be the unique extension of $A_z$, 
which contradicts the hypothesis that it can be a general linear subspace of a quadric containing $A_z$ and 
different from $Q_0$. So there is no such quadric, and we are done. 

\smallskip\noindent {\bf Nonextendable case}.
Now we proceed with a similar analysis for $B_z$ instead of $A_z$. So consider a $\PP^5$ generated by $B_z$ 
and the matrix $N$ as above, and suppose that it meets the Cayley plane along the union of $B_z$ with another
hyperplane. Again we may suppose that $b\ne 0$, and again there should exist a scalar $\gamma$ such that 
$\bar{a}b=\gamma c$, and a linear form $\phi$ on $R_z$ such that $\bar{u}b=\phi(u)c$ for all $u\in R_z$, 
so that we are reduced to the scalar equation $1=\phi(u)+s\gamma$. Again, on the kernel of $\phi$ we get the 
identity $\bar{u}b=0$,  which requires that $b$ is isotropic and $u$ belongs to $L_b$. But now 
it is not impossible that $L_b$ meet $R_z$ along a hyperplane: as we recalled just before starting this proof,
this happens exactly when $b\in R_{\bar z}$. 

\smallskip\noindent {\it First case}. Suppose $a$ is always isotropic. Since $a$ is defined only modulo $R_z$, 
this means the subspace generated by $a$ and $R_z$ is isotropic, so $a$ belongs to $R_z$ since the latter
is maximal, and then we may suppose that $a=0$. Then also $v=w=0$ and our equations reduce to 
$\bar{u}b=\phi(u)c$ for all $u\in R_z$. But observe that if $u=sz$ and $b=t\bar{z}$, then $\bar{u}b=(\bar{z}\bar{s})(t\bar{z})
=\bar{z}(\bar{s}t)\bar{z}$ by the Moufang identity, where the latter product makes sense without further bracketing 
since $\OO$ is alternative (any subalgebra generated by two elements is associative). And since $z$ is isotropic,
$$\bar{z}(\bar{s}t)\bar{z}=(2\langle z,\bar{s}t\rangle -(\bar{t}s)z)\bar{z}=2\langle z,\bar{s}t\rangle \bar{z}$$
is always a scalar multiple of $\bar{z}$. This means that our quadric must be contained in the linear space of 
matrices of the form 
$$\begin{pmatrix} u&a&b \\ \bar{a}&0& c \\ \bar{b}&\bar{c}&0 \end{pmatrix}, \qquad a\in R_z, \; 
b\in R_{\bar{z}},  \; c\in \langle \bar{z} \rangle . $$
This gives a $\PP^9$, which cuts the Cayley plane along the quadric of equation $\bar{a}b=uc$.

\smallskip\noindent {\it Second case}. Now suppose that $a$ can be chosen to be non isotropic, so that the 
equation $\bar{a}b=\gamma c$ implies that $|a|^2b=\gamma ac$, from which we get that $\bar{x}(ac)=2\langle x,a\rangle c
$ for all $x\in R_z$. But since $a$ is only defined modulo $R_z$, this requires that for all $y\in R_z$, 
$\bar{x}(yc)=2\langle x,y\rangle c =0$, implying that in fact $yc=0$ for any $y\in R_z$, hence that $c$ is a multiple 
of $\bar{z}$. We finally deduce that the maximal linear spaces compatible with our conditions consists in 
matrices of the form 
$$\begin{pmatrix} u&a& \theta a\bar{z} \\ \bar{a}&v& \theta v\bar{z} \\  \theta z\bar{a}&\theta v z&0 \end{pmatrix}, \qquad a\in \OO, $$
where $\theta$ is a fixed scalar. The intersection of this linear space with the Cayley plane is the quadric $Q_\theta$ 
of equation $uv=|a|^ 2$, and we get a one parameter family of such quadrics, proving our claim. 
\end{proof}

Note the unexpected consequence that for any eight-dimensional quadric in the Cayley plane, its two families of 
maximal linear spaces can be distinguished concretely: one is made of extendable $\PP^4$'s, the other of 
non extendable ones!

\begin{prop} 
Property (UE) holds for maximal quadrics in the Cayley plane.
\end{prop}

\begin{proof} It suffices to prove that two distinct quadrics $Q$ and $Q'$ always meet in codimension bigger that one. 
Here again we can argue that the dimension of $Q\cap Q'$ is maximal when the corresponding points $q$ and $q'$ 
in $E_6/P_6$ are as "close" as possible, meaning that they are joined by a line in $E_6/P_6$. Up to conjugation, 
we may suppose that $q$ and $q'$ correspond in $\PP(V_{\omega_6})$ to the weight spaces of weights $\omega_6$
and $s_{\alpha_6}(\omega_6)=\omega_5-\omega_6$. The linear spans in $\PP(V_{\omega_1})$ of the corresponding quadrics 
can be obtained by taking the orthogonal to the tangent spaces of $E_6/P_6$ at these points. We readily compute that,
with some abuse of notations, 
$$\begin{array}{rcl} 
\langle Q\rangle & = & \langle \omega_1,  \omega_3- \omega_1, \omega_4- \omega_3,  \omega_2+ \omega_5- \omega_3, 
 \omega_5- \omega_2, \omega_2-\omega_5+ \omega_6, \\
 & &  \omega_4- \omega_2- \omega_5+ \omega_6, \omega_3- \omega_4+ \omega_6,
  \omega_1- \omega_3+ \omega_6,- \omega_1+ \omega_6\rangle , \\
  \langle Q'\rangle & = & \langle \omega_1,  \omega_3- \omega_1, \omega_4- \omega_3,  \omega_2+ \omega_5- \omega_3, 
 \omega_5- \omega_2, \omega_2-\omega_6,  \omega_4- \omega_2- \omega_6, \\
  & & \omega_3- \omega_4+\omega_5-\omega_6,
  \omega_1- \omega_3+\omega_5- \omega_6,- \omega_1+\omega_5- \omega_6\rangle .
  \end{array}$$
Here we just indicated the weights of the weight vectors (recall the multiplicities are all one), and we 
deduced the weights in $ \langle Q'\rangle $ by applying $s_{\alpha_6}$ to those in  $ \langle Q\rangle $.
Finally, we conclude that 
$$Q\cap Q'=\PP\langle \omega_1,  \omega_3- \omega_1, \omega_4- \omega_3,  \omega_2+ \omega_5- \omega_3, 
 \omega_5- \omega_2\rangle $$
 is a $\PP^4$ in the Cayley plane, which in particular confirms our claim. \end{proof} 
 
 \begin{remark}
 Kuznetsov suggested a shorter proof of Proposition \ref{hilbQ-E6P1}, based on the following observations.
 The Cayley  plane $E_6/P_1\subset\PP(V_{\omega_1})$ is a Severi variety whose secant variety is the hypersurface 
 $\mathcal{C}$  defined by the invariant Cartan cubic. The derivatives of this cubic define a birational map 
 $\varphi :\PP(V_{\omega_1})\dashrightarrow \PP(V_{\omega_6})$ described in detail in \cite{ein-sb}, which is the composition of
 the blowup of $E_6/P_1$ with the contraction of the strict transform of $\mathcal{C}$  to the dual Cayley plane 
 $E_6/P_6$. Now, if $Q\subset E_6/P_1$ is a quadric of dimension $m\ge 5$, its linear span $\langle Q\rangle$ cannot 
 be contained in $E_6/P_1$, and since the linear system that defines $\varphi$ is trivial on  $\langle Q\rangle$,
 this linear space must be contracted to a single point. This implies that $\langle Q\rangle\subset \PP(N_y)$
 is contained in the projectivized normal space to $E_6/P_6$ at some point $y$. In particular $m\le 8$, and if
 $m=8$ there must be equality, proving that eight-dimensional quadrics are parametrized by $E_6/P_6$. 
 \end{remark}
 
 \medskip Passing to a general hyperplane section, a naive dimension count would indicate that $F_4/P_4$ should 
contain a six-dimensional family of copies of $\QQ^8$. But since this is family must be $F_4$-invariant, this is 
impossible and it must actually be empty. We conclude:

\begin{prop} 
\label{prop_UE_quadrics_F4/P4}
There is an isomorphism 
$$\Hilb_{\QQ^7}(F_4/P_4)\simeq \Hilb_{\QQ^8}(E_6/P_1)=E_6/P_6.$$
Moreover, maximal quadrics in $F_4/P_4$ have Property (UE). 
\end{prop} 

Using our general strategy, we will thus deduce that automorphisms of $H_x$ can be lifted to $E_6$.

\subsection{$F_4/P_3$, the space of special lines in $F_4/P_4$} 
It is shown in \cite[Proposition 6.7]{lm-proj} that $F_4/P_3$ parametrizes
special lines in $F_4/P_4$. Special means the following: given a point $p$ in $F_4/P_4$, its stabilizer is a parabolic 
group isomorphic to $P_4$, whose Levi part has type $B_3$. As a $Spin_7$-module, the isotropy representation at $p$ (which is nothing else than $\ff_4 /\fp_4$, where $\fp_4$ is the Lie algebra of $P_4$)
is isomorphic to $\Delta_8\oplus V_7$, the sum of the spin and the natural
representation; but  only $\Delta_8$ is invariant under the whole parabolic. Special lines then correspond to points 
inside the closed orbit in $\PP\Delta_8$, that is the spinor variety $\OG(3,7)\simeq\QQ^6$.

Since $F_4/P_3$ parametrizes projective lines in $F_4/P_4$, it is a subvariety of the Grassmannian $G(2,V_{\omega_4})$. Thus its linear spaces must be linear spaces in the Grassmannian, 
hence of two possible types: either spaces of lines contained in a given plane; or spaces of lines passing
through some fixed point, say $p$. The former ones are not extendable. The latter ones are in correspondence
with linear subspaces of $\OG(3,7)\simeq\QQ^6$. In particular, their dimension is at most three, and the three-dimensional 
ones are parametrized by the union of two six-dimensional quadrics. We shall deduce:

\begin{prop} 
$\Hilb_{\PP^3}(F_4/P_3)$ has two connected components, both of dimension $21$ and fibered in quadrics 
over $F_4/P_4$. One is homogeneous and isomorphic with $F_4/P_{3,4}$. The other one is not homogeneous. 
\end{prop} 

\begin{proof} 
The homogeneity is equivalent to the transitivity of the isotropy action on the sets of projective spaces 
made of lines passing through a fixed point $p$. As we have seen, this action factorizes through the action
of $Spin_7$ on $\OG(3,7)\simeq\QQ^6$, and we have to consider the maximal linear spaces on that quadric. Of
course we recover two copies of the same spinor variety, but we have to be careful about the induced action 
of $Spin_7$. If we first think about the same situation for $Spin_8$ acting on $\QQ^6$, the maximal 
linear spaces are parametrized by $\OG(4,8)^+$ and $\OG(4,8)^-$. By triality, we can permute these three spaces
and we deduce that the maximal linear spaces in $\OG(4,8)^+$ are parametrized by $\QQ^6$ and $\OG(4,8)^-$,
with their standard actions of $Spin_8$. Restricting this statement to $Spin_7$, we deduce that the maximal 
linear spaces in $\OG(3,7)$ are parametrized by $\QQ^6$ and $\OG(3,7)$, with their standard actions of $Spin_7$. 
And now, although these are two quadrics of the same dimension, there is a huge difference between the two
actions of $Spin_7$: the second one is transitive, but the first one is not!

This implies that $\Hilb_{\PP^3}(F_4/P_3)$ really has two connected components, among which one and only one is homogeneous;
more precisely, we have seen that the homogeneous component is isomorphic to the variety of pairs of points and 
special lines in $F_4/P_4$, that is $F_4/P_{3,4}$. Note that the other component contains $F_4/P_{1,4}$ as 
a divisor. 
\end{proof} 

This phenomenon has strong similarity with what we observed in Proposition \ref{ext-next}, in the sense that two families 
of maximal isotropic spaces in quadrics, that should in principle be indistinguishable, can in fact be distinguished from
their embedding in the ambient homogeneous space. 

\medskip
Finally, we have enough information to check that our general strategy works fine for smooth hyperplane 
sections $H_x$ of $F_4/P_3$. We just need to check that:

\begin{prop}
\label{prop_F4/P3_UE}
Property (UE) does hold for linear spaces in $F_4/P_3$, in the sense that every extendable projective plane has a unique 
extension in the homogeneous component $F_4/P_{3,4}$ of $\Hilb_{\PP^3}(F_4/P_3)$. 
\end{prop} 

\begin{proof} We have seen in describing these extendable planes that they correspond to planes in six-dimensional
quadrics. Such a plane has a unique extension to either of the two families of three-planes in the quadric,
in particular there is a unique extension to a three-plane from the homogeneous component $F_4/P_{3,4}$.
\end{proof}

\section{Proof of Theorem \ref{thm_1.1}}
\label{sec_proof}

We have now almost all the ingredients to prove Theorem \ref{thm_1.1}, up to a few technical  lemmas whose proofs are postponed to the end of the section. 

First we discard the $G_2$-Grassmannians, since for $G_2/P_1$ Theorem \ref{thm_1.1} is already known from \cite{pz2}, and $G_2/P_2$ is just a quadric. 
For the sake of clarity, we will then distinguish two sets of generalized Grassmannians.  The first set $\Omega_1$ will consist in those for which Property (UE) holds for linear subspaces. By Proposition \ref{prop_UE_property_long_case}, Lemma \ref{lemma_generic_Sp_grassm_UE} and Proposition \ref{prop_F4/P3_UE}, this means that $X=G/P$ is different from $\OG(k,2n+1)$ for $k\leq n-k$, from $\IG(n,2n)$, $F_4/P_1$ and $F_4/P_4$. 
The second set $\Omega_2$ is just the complement of $\Omega_1$ in the set of generalized Grassmannians (which are not $G_2$-Grassmannians). So $\Omega_2$ contains exactly the  Grassmannians $\OG(k,2n+1)$ for $k\leq n-k$, $\IG(n,2n)$, $F_4/P_1$ and $F_4/P_4$. In the previous sections we have shown that for these varieties, Property (UE) holds \emph{for quadrics}. 

Recall that by $m$ (respectively $l$) we denote the maximal dimension of linear subspaces (resp. quadrics) in $X$. If a generalized Grassmannian $X$ 
belongs to $\Omega_1$, let $N_\PP(X)$ denote the set of $\PP^{m-1}$ inside $X$ that cannot be extended to a $\PP^m$ inside $X$. If $X$ belongs to 
$\Omega_2$, let $N_\QQ(X)$ denote the set of $\QQ^{l-1}$ (possibly singular) inside $X$ that cannot be extended to a $\QQ^l$ inside $X$.

\begin{proof}[Conclusion of the proof of Theorem \ref{thm_1.1}]
We follow the general strategy of Section \ref{sec_gen_strategy}. Suppose that the generalized Grassmannian $X=G/P$ belongs to $\Omega_1$, so that  Property (UE) holds for linear subspaces.
In particular, for any hyperplane section $H_x$ we get an extension map 
$$e_x: \Hilb_{\PP^{m-1}}(H_x)^{ext}\longrightarrow \Hilb_{\PP^m}(X),$$
where $\Hilb_{\PP^{m-1}}(H_x)^{ext}\subset \Hilb_{\PP^{m-1}}(H_x)$ denotes the set of $(m-1)$-dimensional linear spaces in $H_x$ that can be extended 
inside $X$ (but not necessarily inside $H_x$). 
Note that this map $e_x$ is an isomorphism outside $\Hilb_{\PP^{m}}(H_x)$. If $f$ is an automorphism of $H_x$, it induces automorphisms of $\Hilb_{\PP^{m}}(H_x)$ and $\Hilb_{\PP^{m-1}}(H_x)$, that we will both denote by $f_*$ for simplicity.
We want to ensure that the second one restricts 
to an isomorphism of $\Hilb_{\PP^{m-1}}(H_x)^{ext}$.

Otherwise said, we want to exclude the possibility that extendable subspaces are sent by $f$ to unextendable subspaces. If there are no unextendable 
spaces  of dimension $m-1$, we are certainly fine. Otherwise $N_\PP(X)\ne \emptyset$, and by Lemma \ref{lem_nonext_nonempty_linear}, $X$ appears in Table 1. 
But then by Lemma \ref{lemma_non_empty-linear}  $\Hilb_{\PP^m}(X)\ne \emptyset$, and since $f_*$ must preserve its pre-image by $e_x$, it has 
to preserve the connected components containing this pre-image, that is precisely $\Hilb_{\PP^{m-1}}(H_x)^{ext}$. Then $f_*$ descends through $e_x$
to a birational automorphism of $\Hilb_{\PP^m}(X)$, extending continuously to the exceptional locus  $\Hilb_{\PP^{m}}(H_x)$; so in fact 
$f_*$ descends to an automorphism of $\Hilb_{\PP^m}(X)$. 

Beware that the latter may have several components. If there is a unique one of a given dimension, say $G/Q$, 
it must also be preserved by $f_*$ and we can then extend $f$ to an element of $Aut(G/Q)=Aut(X)$. If there are several components in each dimension, 
that would be exchanged by $f_*$, then we can use Lemma \ref{lem_different_dim} and compose with an outer automorphism of $X$ to arrive to the same conclusion. 

For the remaining cases where $X$ belongs to $\Omega_2$,
Property (UE) holds for quadrics and the same arguments apply, with the help of Lemmas \ref{lem_nonext_nonempty_quadrics} and \ref{lemma_non_empty-quadrics}.
The only point is that when $X$ is $F_4/P_4$ or $\IG(2,2n)$, hence a hyperplane section respectively of $E_6/P_1$ and $G(2,2n)$, we can 
only conclude that $\Aut(H_x)$ is contained respectively in $E_6$ and $\PGL_{2n}$. For these special cases 
the claims of the Theorem follow from the more precise Theorem \ref{thm_exact_sequence_aut_jordan} proved in the next section.
\end{proof}

Now we turn to the proofs of the technical lemmas used in the previous proof. 

\begin{lemma}
\label{lem_nonext_nonempty_linear}
The set  of generalized Grassmannians $X$ from $\Omega_1$  for which $N_\PP(X)$ is non-empty  is given by Table  \ref{table_non_extendable_components-linear}.
\end{lemma}

\begin{proof}
This is the result of a case by case analysis. If $X$ is of type $A_n$, that is $X=G(k,n+1)$, we can suppose that $2k\leq n+1$. 
Non-extendable linear subspaces are parametrized by $G(k-1,n+1) \cup G(k+1,n+1)$, their respective dimensions being  $n+1-k$ and $k$. 
Therefore $m=n+1-k$ and $N_\PP(X)$ is non-empty if and only if $k=n-k$.

If $X=\OG(k,2n+1)$ with $k>n-k$, by the proof of Proposition \ref{prop_UE_property_long_case}, $m=k$ and $N_\PP(X)\neq \emptyset$ if and only if $k-1=n-k$.

If $X=\IG(k,n)$ for $2\leq k\leq n-1$, by Lemma \ref{lemma_generic_Sp_grassm_UE}, there are two families of non-extendable linear subspaces, of dimensions $k$ and $2n-2k+1$. Thus $N_\PP(X)\neq \emptyset$ if and only if either $k=2n-2k$ or $k-1=2n-2k+1$. Notice that in both cases $k$ needs to be even.

If $X=\OG(k,2n)$ with $k\leq n-1$, by Theorem \ref{thm_subdiagrams_An} there are three families of non-extendable linear subspaces in $X$, of dimensions $k$ and twice $n-k$. Therefore $N_\PP(X)\neq \emptyset$ if and only if either $k=n-k-1$ or $k-1=n-k$. If $X=\OG(n,2n)$ with $n\geq 4$, there are two families of non-extendable linear subspaces in $X$, of dimensions $n-1$ and $3$. Therefore $N_\PP(X)\neq \emptyset$ if and only if $n=5$.

In the exceptional cases one can use a case by case argument based on Theorem \ref{thm_subdiagrams_An}. The only exception is $F_4/P_3$, for which 
we use Proposition \ref{prop_F4/P3_UE} instead.
\end{proof}

\begin{lemma}
\label{lem_nonext_nonempty_quadrics}
The set  of generalized Grassmannians $X$ from $\Omega_2$ for which  $N_\QQ(X)$ is non-empty  is contained in Table    \ref{table_non_extendable_components-quadrics}. 
\end{lemma}

\begin{proof}
If $X=\OG(k,2n+1)$ with $k\leq n-k$, by the content of Section \ref{sec_ort_grass_quadrics}, 
there are three types of quadrics of dimensions $4$, $2n-2k+1$ and $k-1$. Therefore if $N_\QQ(X)\neq \emptyset$ then $4=2n-2k$, hence $k=n=4$.
If $X=\IG(n,2n)$, by Lemma \ref{lem_isotr_grass} $N_\QQ(X)=\emptyset$. 
it is not clear whether $N_\QQ(X)$ is empty or not.
\end{proof}

\begin{table}[]
\centering
\caption{Varieties in $\Omega_1$ for which $N_\PP$ is non-empty}
\label{table_non_extendable_components-linear}
\begin{tabular}{c|c|c|}
Type & $X$  & $\Hilb_{\PP^{m}}(X)$ \\
\hline
 $A_n$ & $G(k,2k+1)$ & $G(k-1,2k+1)$  \\
 $B_n$ & $\OG(k,4k-1)$ & $\OG(k+1,4k-1)$ \\
 $C_n$ & $\IG(k,3k)$ & $\IG(k-1,3k)$  \\
 $C_n$ & $\IG(k,3k-2)$ & $\IG(k+1,3k-2)$ \\
 $D_n$ & $\OG(5,10)_+$ & $\OG(5,10)_-$ \\
 $D_n$ & $\OG(k,4k-2)$ & $\OG(k+1,4k-2)$  \\
 $D_n$ & $\OG(k,4k+2)$ & $\OFl(k-1,2k+1,4k+2)_\pm$  \\
 $E_6$ & $E_6/P_1$ & $E_6/P_2$  \\
 $E_6$ & $E_6/P_3$ & $E_6/P_{1,2}$  \\
 $E_6$ & $E_6/P_4$ & $E_6/P_{2,3}\cup E_6/P_{2,5}$  \\
 $E_7$ & $E_7/P_2$ & $E_7/P_3$  \\
 $E_7$ & $E_7/P_4$ & $E_7/P_{2,3}$  \\
 $E_7$ & $E_7/P_5$ & $E_7/P_{2,6}$  \\
 $E_7$ & $E_7/P_6$ & $E_7/P_{2,7}$  \\
 $E_7$ & $E_7/P_7$ & $E_7/P_2$  \\
 $E_8$ & $E_8/P_5$ & $ E_8/P_{4}\cup E_8/P_{2,6}$  \\
 $E_8$ & $E_8/P_6$ & $E_8/P_{2,7}$  \\
 $E_8$ & $E_8/P_7$ & $E_8/P_{2,8}$  \\
 $E_8$ & $E_8/P_8$ & $E_8/P_2$  \\
 \end{tabular}
 \end{table}

 \begin{lemma}
\label{lem_different_dim}
Let $X=G/P$ be a generalized Grassmannian from $\Omega_1$ (resp.  $\Omega_2$). If $G/Q_1$, $G/Q_2$ are two connected components of $\Hilb_{\PP^{m}}(X)$ (resp. $\Hilb_{\QQ^l}(X)$) of the same dimension, there exists an automorphism of $X$ that exchanges them.
\end{lemma}
  
Notice that an automorphism that exchanges two connected components of $\Hilb_{\PP^m}(X)$ or $\Hilb_{\QQ^l}(X)$ 
must be induced by an outer automorphism of $G$. And recall that such outer automorphisms are detected by the symmetries of the Dynkin diagram.

\begin{proof} We proceed by a case by case analysis, starting with varieties $X$ from $\Omega_1$. If  $X=G(k,n+1)$, with $2k\leq n+1$, non-extendable linear subspaces have dimensions $n+1-k$ and $k$. When $k=n+1-k$, the two components of $\Hilb_{\PP^m}(X)$ are exchanged by an outer automorphism. 
If $X=\OG(k,2n+1)$ with $k>n-k$, by the proof of Proposition \ref{prop_UE_property_long_case}, $\Hilb_{\PP^m}(X)$ is connected.
If $X=\IG(k,2n)$ with $2\leq k\leq n-1$, then by Lemma \ref{lemma_generic_Sp_grassm_UE} there are two families of non-extendable linear subspaces, of dimensions $k$ and $2n-2k+1$. So let us suppose that $k=2n-2k+1$. The two components of $\Hilb_{\PP^m}(X)$ are then $\IG(k+1,3k-1)$ and $\IG(k-1,3k-1)$, and their dimensions are different. 

If $X=\OG(k,2n)$ with $k\leq n-1$, then by Theorem \ref{thm_subdiagrams_An} there are three families of non-extendable linear subspaces in $X$, of dimensions $k$ and twice $n-k$. Moreover the second and  third families are exchanged by an outer automorphism. If $k=n-k$, the components of $\Hilb_{\PP^m}(X)$ are $\OG(k+1,4k)$ and $\OFl(k-1,2k,4k)_\pm$, and the dimensions coincide only for $k=1,2$. However $k=1$ is excluded since  $D_2=A_1\times A_1$, and for 
$k=2$ we get three isomorphic components exchanged by the outer automorphisms of $D_4$. If $X=\OG(n,2n)$ with $n\geq 4$, there are two families of non-extendable linear subspaces, of dimensions $n-1$ and $3$; these dimensions coincide only for $n=4$, in which case $X=\OG(4,8)$ is just a quadric. 

In the exceptional cases one can use a case by case argument and Theorem \ref{thm_subdiagrams_An} to understand when $\Hilb_{\PP^m}(X)$ is not connected. This happens exactly for $E_6/P_2$ and $E_6/P_4$ (which are preserved by an outer automorphism), and $E_8/P_5$. If $X=E_8/P_5$,  $\Hilb_{\PP^m}(X)$ has two connected components $E_8/P_4$ and $E_8/P_{2,6}$, whose dimensions are respectively $106$ and $107$. 

Finally, for the varieties from $\Omega_2$, one checks in a similar way that $\Hilb_{\QQ^l}(X)$ is always connected.
\end{proof}

\begin{lemma}
\label{lemma_non_empty-linear}
Let $X$ be a generalized Grassmannian from $\Omega_1$ appearing in Table \ref{table_non_extendable_components-linear}. 
Then $\Hilb_{\PP^{m}}(H_x)$ is non-empty for any hyperplane section $H_x\subset X$. 
\end{lemma}

\begin{remark}
Notice that this statement cannot be extended to all the generalized Grassmannians. For instance $\Hilb_{\PP^m}(H_x)=\emptyset$ when $H_x$ is a general hyperplane section of $G(2,2n)$.
\end{remark}

\begin{proof}
We first start with a few general remarks. Let us begin with one of the long root spaces in Table 1, say $X=G/P\subset\PP (V)$, where 
$V=V_{\omega_k}$ is some fundamental representation of the simply connected group $G$. Recall that diagrammatically, $X$ corresponds to
a marked Dynkin diagram $(\Delta,\alpha_k)$. Then $\Hilb_{\PP^{m}}(X)$ can be described by 
looking for the subdiagrams of type $(A_m,\alpha_1)$ in this marked diagram. As recalled in Theorem \ref{thm_subdiagrams_An}, the boundary vertices of such 
a subdiagram define a component $G/Q$ of $\Hilb_{\PP^{m}}(X)$. On such a $G/Q$, the category of $G$-equivariant vector bundles is 
equivalent to the category of $Q$-modules \cite{or}. In particular, the minimal $Q$-submodule of the $G$-module $V$ defines a vector bundle $\cE$ on $G/Q$, and the element $x$ of $V^\vee$ that defines $H_x\subset\PP (V)$ also defines a global section of $\cE^\vee$, whose zero locus is precisely $\Hilb_{\PP^{m}}(H_x)$. 
How can we deduce from this description that $\Hilb_{\PP^{m}}(H_x)$ is non empty? First note that by semi-continuity, we just need to prove it 
when $x$ is generic, in which case, since by construction $\cE^\vee$ is globally generated, $\Hilb_{\PP^{m}}(H_x)$ is smooth (and of the expected 
codimension $m+1$ at each of its points, if any). Then we could use the Thom-Porteous formula for its fundamental class and check it is non 
zero in the Chow ring. Or, and that is the approach we will choose, we can use the Koszul complex 
$$  0 \to \wedge^{m+1}\cE  \to \wedge^{m}\cE \to \cdots \to  \cE\to \cO_{G/Q} \to \cO_{\Hilb_{\PP^{m}}(H_x)} \to 0,
$$     
to check that $H^0(\cO_{\Hilb_{\PP^{m}}(H_x)})\ne 0$. Indeed, by a standard diagram chasing it suffices to check that $H^q(G/Q, \wedge^{q}\cE)=0$ 
for any $q>0$. This can readily be deduced from the Bott-Borel-Weil theorem (see e.g. \cite{or})
since $ \wedge^{q}\cE$ is an {\it irreducible} equivariant vector 
bundle for any $q$. Indeed, this follows from the fact that $\cE$, by construction is defined from a representation of $Q$ which is essentially
the natural representation of $SL_{m+1}$; since all the wedge powers of this representation remain irreducible, our claim follows.

We will give more details for the  classical case $X=G(k,2k+1)$, and for the exceptional case $X=E_7/P_3$. The remaining cases can be treated in 
a completely similar way and will be left to the reader. 

\smallskip If $X=G/P=G(k,2k+1)$, then $m=k+1$ and $\Hilb_{\PP^{m}}(X)=G/Q=G(k-1,2k+1)$. Moreover $\cE=\cQ(-1)$ if $\cQ$ denotes the rank $k+2$ 
tautological quotient bundle. This simply yields $ \wedge^{q}\cE=\wedge^{q}\cQ(-q)$. In more Lie theoretic terms we can observe that the 
highest weight of the $Q$-module defining $\cE^\vee$ is $\omega_k$, from which we get the other weights from the action of the Weyl group of 
$Q$, which yields successively $\omega_{k-1}-\omega_{k}+\omega_{k+1}, \ldots \omega_{k-1}-\omega_{2k-1}+\omega_{2k}, \omega_{k-1}-\omega_{2k}$. 
Then we get the weights of $\cE$ by taking their opposites, and we deduce the highest weight of $ \wedge^{q}\cE$ by summing the first $q$,
starting from the highest one: this yields $\omega_{2k+1-q}-q\omega_{k-1}$, for 
$0<q\le k+2$. But then the Bott-Borel-Weil theorem implies that  $ \wedge^{q}\cE$ is acyclic, since for $\rho$ the sum of the 
fundamental weights, $\omega_{2k+1-q}-q\omega_{k-1}+\rho$ is orthogonal to any root $\alpha_i+\cdots +\alpha_j$ for $i\le k-1\le j\le 2k-q$ and 
$j-i+1=q$. Such a root always exists if $q\le k$, while for $q=k+1$ (resp. $q=k+2$) we can choose the root $\alpha_1+\cdots +\alpha_k$ (resp. 
$\alpha_1+\cdots +\alpha_{k+1}$).

\smallskip If $X=G/P=E_7/P_3$, then $m=5$ and $\Hilb_{\PP^{m}}(X)=G/Q=E_7/P_{1,2}$. The highest weight of $\cE^\vee$ is $\omega_3$, and 
as before we get the other weights from the action of the Weyl group of $Q$, which yields successively $\omega_1-\omega_3+\omega_4, 
 \omega_1+\omega_2-\omega_4+\omega_5, \omega_1+\omega_2-\omega_5+\omega_6, \omega_1+\omega_2-\omega_6+\omega_7, \omega_1+\omega_2-\omega_7$. 
 Then we get the weights of $\cE$ by taking their opposites, and we deduce the highest weight of $ \wedge^{q}\cE$ by summing the first $q$,
starting from the highest one: this yields successively $\omega_7-\omega_1-\omega_2, \omega_6-2\omega_1-2\omega_2, \omega_5-3\omega_1-3\omega_2, \omega_4-4\omega_1-4\omega_2, \omega_3-5\omega_1-4\omega_2, -5\omega_1-4\omega_2$. As before, we conclude by the Bott-Borel-Weil that  $ \wedge^{q}\cE$ is acyclic for any $q>0$, since adding $\rho$ we get a weight which is orthogonal respectively to the root $\alpha_1$, $\alpha_1+\alpha_3$, $\alpha_1+\alpha_3
+\alpha_4$, $\alpha_1+\alpha_3+\alpha_4$, $\alpha_1+\alpha_3+\alpha_4+\alpha_5$, $\alpha_1+\alpha_3+\alpha_4+\alpha_5+\alpha_6$. (For another proof, 
more conceptual but that cannot be transposed to all the other cases, one can observe that the projection $E_7/P_{1,2}\to E_7/P_2$ is a projective 
bundle, and that $\cE$ restricts on the fibers of this bundle to $\cQ^\vee(1)$, the twisted dual of the quotient bundle on projective space. 
And since $\cQ^\vee(1)$ and its positive wedge powers are acyclic on projective space, $\cE$  and its positive wedge powers must be acyclic 
as well on $G/Q$.)

\smallskip Let us now turn to the short root cases. We will discuss in some details the case where  $X=\Sp_{3k}/P_k=\IG(k,3k)$, with $k=2h$ even, the other one
being similar. We will suppose that $h\geq 2$, leaving the case of $X=\IG(2,6)$ to the reader. Then $m=k+1$ and $\Hilb_{\PP^m}(X)=\Sp_{3k}/P_{k-1}=\IG(k-1,3k)$. Moreover, for any hyperplane section $H_x$ of $X$, $\Hilb_{\PP^m}(H_x)\subset \Hilb_{\PP^m}(X)$ is the zero locus of a section of the rank $k+2$ vector bundle  $\cE^\vee=(\cU^\perp/\cU)(1)$, where $\cU$ denotes the rank $k-1$ tautological bundle on $\IG(k-1,3k)$ and $\perp$ is taken with respect to the skew-symmetric form fixed by $\Sp_{3k}$. Here the main difference with the long root cases is that the representation that defines the homogeneous bundle $\cU^\perp/\cU$ is essentially the natural representation $M$ of $Sp_{k+2}$, whose wedge powers are {\it not} irreducible. 
Indeed, recall that if $\omega_M$ is the invariant symplectic form on $M$, and $\Omega_M\in\wedge^2M$ the dual bivector, then for any $q>1$ the $q$-th wedge power $\wedge^qM$ contains $\wedge^{q-2}M\wedge \Omega_M$, of which the $q$-th fundamental representation is a supplement, sometimes denoted 
$\wedge^{\langle q\rangle}M$. One deduces inductively the decomposition $\wedge^qM= \wedge^{k+2-q}M=\oplus_{p\ge 0}\wedge^{\langle q-2p\rangle}M$
for $1\le q\le h+1$. 
In terms of our vector bundles, we get 
$$\wedge^q\cE=\bigoplus_{p\ge 0}\wedge^{\langle q-2p\rangle}(\cU^\perp/\cU)(-q)\quad \mathrm{for}\; q\le h+1,$$
$$\wedge^q\cE=\bigoplus_{p\ge 0}\wedge^{\langle k+2-q-2p\rangle}(\cU^\perp/\cU)(-q)\quad \mathrm{for}\; q>h+1.$$
In order to apply the Bott-Borel-Weil theorem, we need to identify the highest weight of each irreducible bundle in these
decompositions. For this we observe that the weights of $\cU^\perp/\cU$ are, 
in decreasing order, $\omega_k-\omega_{k-1}, \ldots , \omega_{3h}-\omega_{3h-1}$ and their opposites. We deduce that for $q\le h+1$,
the highest weight of $\wedge^{\langle q\rangle}(\cU^\perp/\cU)$, or $\wedge^q(\cU^\perp/\cU)$, is the sum of the first $q$
of those weights, that is $\omega_{k+q-1}-\omega_{k-1}$. So the highest weight of $\wedge^{\langle q-2p\rangle}(\cU^\perp/\cU)(-q)$ is 
$\theta=\omega_{k+q-2p-1}-(q+1)\omega_{k-1}$, and then $\theta+\rho$ is orthogonal to the roots $\alpha_i+\cdots +\alpha_j$ when $i\le k-1\le j<k+q-2p-1$
and $j-i=q$, or $i\le k-1\le k+q-2p-1\le j$ and $j-i=q-1$. Similarly, for $q>h+1$ the  highest weight of $\wedge^{\langle k+2-q-2p\rangle}(\cU^\perp/\cU)(-q)$
is $\theta=\omega_{2k+1-q-2p}-(q+1)\omega_{k-1}$, and $\theta+\rho$ is orthogonal to the root $\alpha_{3h-q+1}+\cdots +\alpha_{3h}$. This implies that 
$\wedge^q\cE$ is acyclic for any $q>0$, and we are done. 
\end{proof}

\begin{table}[]
\centering
\caption{Varieties in $\Omega_2$ for which $N_\QQ$ is possibly non-empty}
\label{table_non_extendable_components-quadrics}
\begin{tabular}{c|c|c|c|c}
Type & $X$  & $ \Hilb_{\QQ^{l}}(X)$  \\
\hline
 $B_4$ & $\OG(2,9)$ & $\QQ^7$ \\
 $F_4$ & $F_4/P_1$ &  $F_4/P_4$   \\
 $F_4$ & $F_4/P_4$ &  $E_6/P_6$ \\
 \end{tabular}
 \end{table}
 
\begin{lemma}
\label{lemma_non_empty-quadrics}
Let $X$ be a generalized Grassmannian from $\Omega_2$ appearing in Table \ref{table_non_extendable_components-quadrics}. Then $\Hilb_{\QQ^l}(H_x)$ 
is non-empty for any hyperplane section $H_x\subset X$.
\end{lemma}

\begin{proof} 
The proof is similar to that of the previous Lemma, except for $X=F_4/P_4$. Indeed this is the only case for which $\Hilb_{\QQ^l}(H_x)$ is not the zero locus of a section of a vector bundle inside $\Hilb_{\QQ^l}(X)$. Instead, recall that seven dimensional quadrics inside $F_4/P_4$ are obtained from eight dimensional quadrics $Q_p$ inside $E_6/P_1$ by intersecting with the hyperplane $H_0$ such that $F_4/P_4=E_6/P_1\cap H_0$. Moreover $Q_p=\PP(\cS_p)\cap (E_6/P_1)$, where $\cS$ is the vector bundle of rank ten on $\Hilb_{\QQ^8}(E_6/P_1)=E_6/P_6$ defined by the weight $\omega_1$. 

So let us consider another general hyperplane section $H_1\subset E_6/P_1$, and $H_x:=H_0\cap H_1\subset F_4/P_4$. The double hyperplane section $H_x$ contains $Q_p\cap H_0$ if and only if the equations of $H_0$ and $H_1$ are proportional when restricted to $\PP(\cS_p)$. If $H_0$ and $H_1$ are defined by sections $x_0$ and $x_1$ in $\Ho^0(E_6/P_6,\cS)$, this means that $\Hilb_{\QQ^7}(H_x)$ 
is the degeneracy locus defined by these two sections. These sections being general, 
its structure sheaf is   resolved by the  Eagon-Northcott complex
$$  0 \to (\wedge^{10}\cS^\vee)^{\oplus 9} \to \cdots \to (\wedge^{i}\cS^\vee)^{\oplus i-1} \to \cdots \to \wedge^2\cS^\vee \to \cO_{E_6/P_6} \to \cO_{\Hilb_{\QQ^7}(H_x)} \to 0.
$$     
Each bundle $\wedge^{i}\cS$ is irreducible, and using the Bott-Borel-Weil theorem as before we deduce that $\Ho^0( \cO_{\Hilb_{\QQ^7}(H_x)})= \CC$. As a consequence, $\Hilb_{\QQ^7}(H_x)\neq \emptyset$.
\end{proof}

\section{Jordan algebras and automorphisms}
\label{sec_type_Jordan}

In this section we focus our attention on some special generalized Grassmannians, for which our general strategy is not conclusive, and we will explain why. These are the symplectic Grassmannians $\IG(2,2n)$ and the exceptional variety $F_4/P_4$. However we will also include in our analysis $v_2(\QQ^{n-2})$ (the second quadratic Veronese embedding of a smooth quadric) and the adjoint
variety $X_{\ad}(\fsl_n)=\Fl(1,n-1,n)$, which is not strictly speaking a generalized Grassmannian: all of them can be described in a uniform way.
Indeed, the common property of all these varieties is that they  can be seen as
hyperplane sections of some bigger homogeneous varieties, that turn out to appear naturally in the context of Jordan algebras. 

Notice moreover that these varieties are the coadjoint varieties of the corresponding groups (in type $A$, coadjoint and adjoint coincide). We recall that coadjoint varieties are obtained in the following way: take the marked Dynkin diagram of an adjoint variety and reverse the double (or triple) arrows; the marked Dynkin diagram thus obtained is the coadjoint variety of the reversed Dynkin diagram. Thus, adjoint and coadjoint varieties coincide in the simply laced case, while in the non-simply laced case we obtain the above mentioned varieties (plus the coadjoint of $G_2$, which is just $\QQ^5$).

\subsection{Hermitian Jordan algebras}

Let us start by recalling a few classical facts about Jordan algebras, and in particular simple complex Jordan algebras coming from Hermitian matrices; we refer to \cite{Springer_Jordan_Algebras} for details.

Let $\mathbf{A}$ be a real finite dimensional normed algebra, so that $\mathbf{A}$ is either $\mathbf{R}, \mathbf{C}, \mathbf{H}$ 
or $\mathbf{O}$ and admits 
a natural conjugation. We will denote by $\AA := \mathbf{A} \otimes_{\RR}\CC$ the complexification of $\mathbf{A}$. Then we will 
consider the space of Hermitian $n\times n$-matrices
\[
\cJ_n(\AA):= \{M\in \cM_{n}(\AA) \mid M^*=M\},
\]
which we will denote simply by $\cJ_n$ if no confusion can arise. Here $M^*$ denotes the transpose conjugate of $M$. We will always assume that $n\geq 3$, and when $\AA=\OO$ we will further restrict to the case $n=3$. These spaces of matrices get a Jordan algebra structure by considering the symmetrized product of matrices. 

On Jordan algebras there exists a rational function $j:\cJ_n \dashrightarrow \cJ_n$ which allows to define the structure group  $$\cG:=\{g\in \GL(\cJ_n) \mid \exists h\in \GL(\cJ_n), g\circ j = j\circ h\}.$$ Actually, a Jordan structure on a vector space is defined starting from such a rational function, which is required to satisfy some additional properties. The group $\cG$ fixes a well-defined degree-$n$ \emph{norm} on $\cJ_n$, which we denote by $\det$. Every element $e\in \cJ_n $ such that $\det(e)\neq 0$ belongs to the $\cG$-orbit of the identity matrix $\mathbbm{1}$. We will denote by $G$ the connected component of $ \Stab_{\cG}(e)$ (one should think of 
$e$ as \emph{the} identity inside $\cJ_n$; when $e=\mathbbm{1}$, $G$ is essentially the automorphism group of the Jordan algebra). Both $\cG$ and $G$ are reductive Lie groups.

\begin{remark}
Each $x\in \cJ_n$ 
satisfies an equation  $x^n-t(x)x^{n-1}+ \cdots\pm\det(x)\mathbbm{1}=0$. So when $\det(x)\neq 0$ one defines the Jordan structure $j$ by the formula $j(x)=x^{-1}= (x^{n-1}-t(x)x^{n-1}+\cdots )/\det(x)$. When
$\mathbf{A}$ is associative this coincides with the usual matrix inverse.
\end{remark}

The minimal $\cG$-orbit inside $\PP(\cJ_n)$ is a homogeneous projective variety $\cX$, sometimes called a Scorza variety
\cite{zak, chaput}. Moreover one can view $e$ as a general section of $\cO_{\cX}(1)$ since $\Ho^0(\cX, \cO_{\cX}(1))\cong \cJ_n^\vee \cong \cJ_n$ (a natural duality is obtained by polarizing the determinant at the identity, which yields the 
non degenerate quadratic form $\det(e,\ldots,e, -, -)$). The zero locus $X\subset \cX$ of $e$ is then a $G$-homogeneous variety, more precisely it is one of the varieties mentioned in the introduction to this section; it is embedded inside the projectivization of $\cJ_n^0=e^\perp \cong \cJ_n/\CC e$. Table \ref{table_jordan:algebras} contains explicit information about $\cJ_n$, $\cG$, $\cX$, $G$, $X$. We denoted by  $\cS_n$ (respectively $\cA_{2n}$) the space of symmetric (resp. skew-symmetric) matrices.

\begin{table}[]
\centering
\caption{Hermitian Jordan algebras and their geometry}
\label{table_jordan:algebras}
\begin{tabular}{c|c|c|c|c|c}
$\AA$ & $\cJ_n(\AA)$  & $\cG/Z(\cG)$ & $\cX$ & $G$ & $X$ \\
\hline
 $\RR$ & $\cS_n$ & $\PGL_n$ & $v_2(\PP^{n-1})$ & $\SO_n$ & $v_2(\QQ^{n-2})$ \\
 $\CC$ & $\cM_n$ & $\PGL_n\times \PGL_n$ & $\PP^{n-1}\times \PP^{n-1}$ & $\PGL_n$ & $\Fl(1,n-1,n)$ \\ 
 $\HH$ & $\cA_{2n}$ & $\PGL_{2n}$ & $G(2,2n)$ & $\Sp_{2n}$ & $\IG(2,2n)$ \\ 
 $\OO$ & $\cJ_{3}(\OO)$ & $E_6$ & $E_6/P_1$ & $F_4$ & $F_4/P_4$  \end{tabular}
 \end{table}
 
 An important consequence of this construction is that 
 a hyperplane section $H_x\subset X$ can be seen as the double hyperplane section inside $\cX$ defined by the pencil $L_x:=\PP(\langle e,x \rangle)\cong \PP^1\subset \PP(\cJ_n)$.  We will start by giving a simple 
 criterion for the smoothness of $H_x$, but before doing so, let us briefly describe more concretely the geometry behind Table \ref{table_jordan:algebras} when $\mathbf{A}=\mathbf{R},\mathbf{C}$ or $\mathbf{H}$:
\begin{itemize}
\item[$\mathbf{A}=\mathbf{R}$] The Jordan algebra $\cJ_n(\RR)$ is the space of symmetric matrices $\cS_n$. The group $\cG/Z(\cG)$ is $\PGL_n$, acting by congruence on $\cS_n$. As a generic element $e\in \cS_n$ one can take the identity matrix; then $\Stab_\cG^0(e)$ is $\SO_n$. The minimal $\cG$-orbit inside $\PP(\cS_n)$ is the space of rank-one symmetric matrices, i.e. the second Veronese embedding $v_2(\PP^{n-1})$, and its hyperplane section is $v_2(\QQ^{n-2})$; $H_x$ is thus the intersection of two quadrics. The function $\det$ is just the determinant;
\item[$\mathbf{A}=\mathbf{C}$] The Jordan algebra $\cJ_n(\CC)$ is the space of matrices $\cM_n$. The group $\cG/Z(\cG)$ is $\PGL_n\times \PGL_n$ with its natural action on $\cM_n$ (the first $\PGL_n$ acts on the left on $\cM_n$ and the second $\PGL_n$ on the right). If $e\in \cM_n$ is the identity matrix then $\Stab_\cG^0(e)$ is $\PGL_n$, embedded in $\PGL_n\times \PGL_n$ by $M \mapsto (M,M^{-1})$. The minimal $\cG$-orbit inside $\PP(\cM_n)$ is the space of rank-one matrices, i.e. $\PP^{n-1}\times \PP^{n-1}$, and its hyperplane section is $\Fl(1,n-1,n)$; $H_x$ is thus the intersection of two $(1,1)$-divisors in $\PP^{n-1}\times \PP^{n-1}$. The function $\det$ is just the determinant;
\item[$\mathbf{A}=\mathbf{H}$] The Jordan algebra $\cJ_n(\HH)$ is the space of skew-symmetric matrices $\cA_{2n}$. The group $\cG/Z(\cG)$ is $\PGL_{2n}$, acting by congruence on $\cA_{2n}$. If $e\in \cA_{2n}$ is non-degenerate, then $\Stab_\cG^0(e)$ is isomorphic to $\Sp_{2n}$. The minimal $\cG$-orbit inside $\PP(\cA_{2n})$ is the space of rank-two skew-symmetric matrices, i.e. the Grassmannian $G(2,2n)$, and its hyperplane section is the isotropic Grassmannian $\IG(2,2n)$; $H_x$ is thus a bisymplectic Grassmannian of planes (as it was defined in \cite{vlad}). The function $\det$ is just the pfaffian.
\end{itemize}

\subsection{Smoothness of (double) hyperplane sections}

We have recalled that as soon as $\det(e)\neq 0$, the zero locus $X$ of $e$ is a $G$-homogeneous projective variety,
in particular smooth. As a matter of fact this is also a necessary condition for smoothness. Indeed, let us denote 
\[
i(e):= \det(e,\ldots,e, -)\in \cJ_n^\vee \cong \cJ_n.
\]
Let us suppose that $\det(e)=0$ and $e$ is generic with this property; then $i(e)\in e^\perp$, and $i(e)$ belongs to $\cX$
(this is a generalization of the comatrix, see \cite{zak}). Moreover one can check that $i(e)$ is actually a singular point (in fact
the only singular point) of the zero locus $X$ of $e$. Thus the hypersurface $\{\det=0\}$ is the dual variety of $\cX$ (a classical fact for determinantal varieties). The natural generalization to double hyperplane sections is the following result, where we suppose that $\det(e)\neq 0$.

\begin{prop}
\label{prop_crit_smooth_double_hyperplane}
The double hyperplane section $H_x$ associated to the pencil $L_x$ is smooth if and only if the intersection of $L_x$ and $\cX^\vee=\{ \det=0 \}$ is transverse, i.e. if and only if $L_x\cap \cX^\vee$ consists of exactly $n$ distinct points.
\end{prop}

\begin{proof}
The pencil $L_x$ is non transverse to $\cX$ at $[y]=[ae+bx]$, $b\neq 0$ if and only if $\det(y)=\det(y,\ldots,y,e)=0$. This happens if and only if $[i(y)]$ is a singular point of the zero locus of $y$ inside $\cX$ (corresponding to the condition $\det(y)=0$, as we have seen before) and $[i(y)]\in X$ (corresponding to $\det(y,\ldots,y,e)=0$). In order to conclude, observe that a point of $H_x$ is singular if and only if it is singular as a point of the zero locus inside $\cX$ of a certain hyperplane section of the pencil $L_x$.
\end{proof}

Note that in the case of the intersection of two quadrics, there is a natural double cover of $L_x$ branched over $L_x\cap \cX^\vee$, giving 
a hyperelliptic curve $C$ whose Jacobian is the intermediate Jacobian of $H_x$ \cite{reid}.

\smallskip
Let us finally state a useful consequence of the previous statement.

\begin{lemma}
\label{lem_diagonalization}
If $H_x$ is smooth, there exist scalars $x_1,\ldots, x_n$, with  $x_i\neq x_j$ for $i\ne j$, 
such that the pencil $L_x$ belongs to the  $\cG$-orbit of $\PP(\langle \mathbbm{1}, \diag(x_1,\ldots,x_n)\rangle)$.
\end{lemma}

\begin{proof}
If $H_x$ is smooth then the characteristic polynomial of $x$ has $n$ distinct roots. The statement is then classical
when $\AA=\RR$ or $\AA=\CC$, and it is a direct consequence of \cite[Theorem 3.1]{Kuznetsov_kuchle_varieties}  when 
 $\AA=\HH$ and \cite[Theorem 0.2]{Nishio_Yasukura_orbit_Jordan_matrix_algebra} when  $\AA=\OO$.
\end{proof}

\subsection*{Automorphisms}

Let us now consider the automorphism group of a smooth (double) hyperplane section $H_x$. The first observation is that we can characterize $\Aut(H_x)$ as follows:

\begin{prop}
$\Aut(H_x)=\Stab_{\Aut(\cX)}(L_x)\subset \Aut(\cX)$.
\end{prop}

\begin{proof}
The equality is a consequence of the inclusion $\Aut(H_x)\subset \Aut(\cX)$. The inclusion is a classical result when $H_x\subset \Fl(1,n,n+1)\subset \PP^{n-1}\times \PP^{n-1}$ and $H_x\subset v_2(\QQ^{n-2})\subset v_2(\PP^{n-1})$. Finally, in Section \ref{sec_type_C} we showed that $\Aut(H_x)\subset \PGL_{2n}$ when $H_x\subset  \IG(2,2n) \subset G(2,2n)$, while in Section \ref{sec_more_on_type_F} we showed that $\Aut(H_x)\subset E_6$ when $H_x\subset  F_4/P_4 \subset E_6/P_1$.
\end{proof}

We deduce the short exact sequence of groups
\[
1\to K \to \Aut(H_x) \to I \to 1,
\]
where the second map is the natural restriction from $\Aut(H_x)$ to $\Aut(L_x)=\PGL_2$, with image $I\subset \PGL_2$. Since $\id\in \Aut(L_x)$ fixes $\CC e$, the group $K$ is contained inside $\Aut(X)$. Moreover since $\cG$ stabilizes $\det$, the group $I$ permutes the $n$ distinct points given by $L_x\cap \cX^\vee$ (see Proposition \ref{prop_crit_smooth_double_hyperplane}). Since we assumed that $n\geq 3$, $I$ is a subgroup of the permutation group $\fS_n$ and thus a finite subgroup of $\PGL_2$.

\begin{lemma}
$I=\{g\in \PGL(L_x) \mid g(L_x\cap \cX^\vee)=L_x\cap \cX^\vee\}$.
\end{lemma}

\begin{proof}
By Lemma \ref{lem_diagonalization}, if $H_x$ is smooth we can assume that $e=\diag(1,\ldots,1)$ and $x=\diag(x_1,\ldots,x_n)$. Let $e_1,\ldots,e_n$ be the standard basis on $\AA_\CC^n$ and fix $\sigma\in \fS_n$. Let $\lambda=(\lambda_1,\ldots,\lambda_n)\in \CC^n$ and denote by $g_\sigma(\lambda)\in\cM_n(\AA_\CC)$ the linear map sending $e_i$ to $\lambda_i e_{\sigma(i)}$. The endomorphism $\hat{g}_\sigma(\lambda)$ of $\cM_n(\AA_\CC)$ sending $M$ to 
$g_\sigma(\lambda)  M  g_\sigma(\lambda)^*$, belongs to $\cG$, and it stabilizes $L_x$ if and only if there exist scalars $\alpha,\beta,\gamma,\delta$ such that for all $i$,
\[
\lambda_i^2=\alpha +\beta x_{\sigma(i)} \qquad \mathrm{and} \qquad x_i\lambda_i^2=\gamma +\delta x_{\sigma(i)}.
\]
Indeed, $\hat{g}_\sigma(\lambda)(e)$ is the diagonal matrix whose $\sigma(i)$-th entry is $\lambda_i^2$, and $\hat{g}_\sigma(\lambda)(x)$ is the diagonal matrix whose $\sigma(i)$-th entry is $x_i\lambda_i^2$, and we are requiring that both belong to $L_x$. Moreover, if this happens, then the restriction of $\hat{g}_\sigma(\lambda)$ to $\PGL(L_x)$ acts on $L_x\cap \cX^\vee$ as the permutation $\sigma$. Vice versa, if $\alpha,\beta,\gamma,\delta$ are the coefficients of an element in $\PGL(L_x)$ acting as $\sigma$ on $L_x\cap \cX^\vee$, define $\lambda$ by the property that for any $i$, $\lambda_i^2=\alpha +\beta x_{\sigma(i)}$; then $\hat{g}_\sigma(\lambda)|_{L_x}$ belongs to $I$.
\end{proof}

The finite subgroups of $\PGL_2$ are well-known: recall they are the cyclic groups $\ZZ_d$, the dihedral groups $\fD_{2m}$, the tetrahedral group $A_4$, the octahedral group $\fS_4$, and the icosahedral group $A_5$. Therefore $I$ is one of these groups. 
Note moreover that among those, only the cyclic groups admit a fixed point inside $\PP^1$.
On the one hand, if $I$ fixes $\CC e$, then $I$ is cyclic  and in such a situation $\Aut(H_x)=\Stab_{\Aut(X)}(\CC x)$ is contained 
in $\Aut(X)$. On the other hand, if $I$ is not cyclic, then $\Aut(H_x)$ is not contained inside any copy of $\Aut(X)$. We summarize our discussion as follows:

\begin{theorem}
\label{thm_exact_sequence_aut_jordan}
Let us suppose that $H_x$ is smooth, or equivalently that $L_x\cap \cX^\vee$ consists of $n$ distinct points. Then there is  an exact sequence 
\[
1\to \Stab_{\Aut(X)}(x) \to \Aut(H_x)\to I \to 1,
\]
where $I=\{g\in \PGL(L_x) \mid g(L_x\cap \cX^\vee)=L_x\cap \cX^\vee\}\subset \fS_n$ is isomorphic to one of the groups $\ZZ_d,\fD_{2m},A_4,\fS_4,A_5$. 
\end{theorem}

Note that $I$ is one of these groups if and only if $L_x\cap \cX^\vee$ is a union of orbits of $I$, and of no bigger group in the list. 
For example, the icosahedral group $A_5$ acting on $\PP^1$ has generic orbits of size $60$, plus three special orbits 
of size $12,20,30$. Taking a union of some of the special orbits, plus any number of generic orbits, we can 
deduce for each value of $n$ whether there exists a smooth hyperplane section with icosahedral symmetry: the result is
stated in the Introduction.

 \medskip
Finally, we can conclude that in some cases, some {\it unnatural automorphisms} can exist, unnatural in the sense that do not lift to $\Aut(X)$. 

\begin{prop}
\label{lem_unnatural_aut_jordan}
For any choice of $n$ distinct points $p_1,\ldots,p_n$ inside $\PP^1$, there exists $x$ such that $H_x$ is 
smooth and $L_x\cap \cX^\vee= \{p_1,\ldots,p_n\}\subset \PP^1$.

If the associated subgroup $I$ of $\PGL(L_x)$ has no fixed point, then the automorphisms of $H_x$ cannot always 
be lifted to automorphisms of $X$, nor of any hyperplane section of $\cX$. 
\end{prop} 

This was already observed in \cite{pv} for the {\it general} hyperplane sections of $\IG(2,6)$ and $\IG(2,8)$. More recently, 
hyperplane 
sections of the symplectic Grassmannians $\IG(2,2n)$ were coined {\it bisymplectic Grassmannians} and studied in greater details in
\cite{vlad}.

\section{Adjoint varieties} 
\label{sec_adjoint}

For a simple complex Lie algebra $\fg$, the adjoint variety $X_{\ad}(\fg)$ is the unique closed $G$-orbit 
inside $\PP(\fg)$, where $G=\Aut(\fg)$ is the adjoint group (or any connected Lie group with Lie algebra $\fg$). 
Although they are not always generalized Grassmannians (type A is the exception), or when they are, not 
necessarily minimally embedded (type C is the exception), adjoint varieties have several nice common properties 
that we will first discuss. 

\subsection{Preliminaries} First observe that by Demazure's results (Theorem \ref{thm_Demazure}),
$$\Aut(X_{\ad}(\fg))=\Aut(\fg)=G\rtimes\Gamma,$$
where $G=\Aut^0(\fg)$ is the adjoint group of $\fg$ and $\Gamma$ is the symmetry group of the associated Dynkin diagram. 

In the rest of this section we will exclude type $C$, for which a hyperplane section of the 
adjoint variety is just a quadric. We will also exclude type A, since the adjoint varieties in this type
are flag manifolds $\Fl(1,n-1,n)$ and were already considered before.

The results of the previous sections imply the following statement:

\begin{prop} 
\label{prop_containment} 
For any smooth hyperplane section $H_x$,
defined by $[x]\in\PP(\fg)$, 
$$\Aut (H_x)=\Stab_{\Aut(\fg)}([x])\subset \Aut(\fg).$$
\end{prop} 

\begin{proof}
We have seen that an element of $\Aut (H_x)$ can be lifted to $\Aut(\fg)$, hence more precisely to 
$\Stab_{\Aut(\fg)}([x])$. There just remains to check that the restriction morphism from $\Stab_{\Aut(\fg)}([x])$
to $\Aut (H_x)$ is injective. 

So consider $g\in \Aut(\fg)$ stabilizing $[x]\in \PP(\fg)$ and acting trivially on $H_x$, hence  as a homothety on $x^\perp$. If $x\notin x^\perp$ then $g$ is semisimple. If $x\in x^\perp$, consider an element $y\notin x^\perp$ such that $(y,y)=0$ (if necessary, add to $y$ a multiple of $x$). Since  $g$ preserves the Killing form and it acts as a homothety on $x^\perp$, the space $x^\perp \cap y^\perp \subset x^\perp$ is stabilized by $g$, and so is its orthogonal $\langle x,y\rangle$. Moreover $[x]$ and $[ y]$ are the only points in $\PP(\langle x,y\rangle)$ whose Killing norm is zero; $g$ stabilizes $\CC x$ thus it also stabilizes $\CC y$, and therefore we can again conclude that $g$ is semisimple. 

As a consequence, the centralizer of $g$ inside $G\rtimes \Gamma$ contains a torus $T$ whose rank is the same as $G$ (even though a priori it is not contained in the connected component of the identity of $G\rtimes \Gamma$). Then $g$ acts as the identity on the Lie algebra of $T$, whose dimension is greater than one, which implies that it acts as the identity on $x^\perp$. Since $G$ is simple, $\det(g)=1$ and $g$ acts as the identity on $\PP(\fg)$, thus $g=\id$. 
\end{proof} 

An interesting and important property of adjoint varieties is that they admit {\it contact structures};
according to the celebrated LeBrun-Salamon conjecture, they should even be the only Fano varieties of Picard number 
one (this excludes type A) admitting such a structure (see e.g. \cite{beauville} for an introduction). Recall that 
a contact structure is given by a special distribution of tangent hyperplanes. Since the adjoint variety is homogeneous, 
its tangent space at $[Z]$ is simply $[\fg, Z]$, and the contact hyperplane is 
$$P_{[Z]}=[Z^\perp, Z]\subset [\fg, Z]=T_{[Z]}X_{\ad}(\fg),$$
where $Z^\perp\subset \fg$ is the orthogonal to $Z$ with respect to the Killing form $K$.
Recall that this contact hyperplane admits a non degenerate skew-symmetric form, defined up to scalar
by $$\omega_Z(U,V)=K(Z,[U,V]).$$ 
This is nothing else than the projective version of the famous Kostant-Kirillov
form. Except in type $A$, the isotropy representation is in fact irreducible on the contact hyperplane, 
and the unique closed orbit 
defines a {\it Legendrian cone}  $LC_{[Z]}\subset \PP(T_{[Z]}X_{\ad}(\fg))$ (see section 3.2 of \cite{lm-classification}). Geometrically, this is the union 
of the lines in the adjoint variety passing through $[Z]$ (that some authors will rather see as the VMRT, or Variety 
of Minimal Rational Tangents, at this point, see e.g. \cite{fh, bfm}).

Now if $x\in \fh$ is regular semisimple with $\fh$ a Cartan subalgebra of $\fg$, the hyperplane section
$H_x$ contains the Legendrian cone $LC_{[Z]}$
exactly when $[Z]$ is a root space of a long root. Indeed, since the linear span of $LC_{[Z]}$ is $[Z^\perp, Z]$, by the usual invariance property of the Killing form 
it is contained in $H_x$ if and only if 
$0=K([Z^\perp,Z],x)=K([x,Z],Z^\perp)$, which implies the claim.  In particular, any automorphism of $H_x$ permutes the points $[Z]$ at which the VMRT coincides with $LC_{[Z]}$ (i.e. the long root spaces), and thus permutes the Legendrian cones. This was used in \cite{pz2} to help understanding the finite part of $\Aut(H_x)$, but we will
follow another route. Note nevertheless the interesting geometric property that any hyperplane section $H_x$ contains
some of these nice Legendrian cones, studied in detail in \cite{buczynski}.

\subsection{Tevelev's formula} 
In order to be able to describe $\Aut (H_x)$ more precisely, we will need to know for which $x\in\fg$ the hyperplane section 
$H_x$ is smooth. The answer to this question is provided by Tevelev's formula \cite{tevelev}. 

Indeed, as always the projective dual variety parametrizes the singular hyperplane sections. 
For the adjoint variety  $X_{\ad}(\fg)$ this is a hypersurface 
$X_{\ad}(\fg)^\vee$ in the dual projective space $\PP(\fg^\vee)$ and Tevelev's formula gives an equation 
for this hypersurface. 

Note that $\PP(\fg^\vee)$ is naturally identified to $\PP(\fg)$ by the
Killing form. Being $G$-invariant, an equation of the dual is given by some $G$-invariant homogeneous polynomial on $\fg$. 
Once we fix a Cartan subalgebra $\fh$ of $\fg$, we get a root space decomposition, a root system $R$ made of linear 
forms on $\fh$, and a Weyl group $W$ that permutes the roots. 
By Chevalley's Theorem, the algebra of $G$-invariant polynomials on $\fg$ is isomorphic to 
 the algebra of $W$-invariant polynomials on $\ft$. 
 
 \begin{prop}[Tevelev's formula]
 The invariant polynomial 
 $$P_\fg = \prod_{\alpha\in R_\ell}\alpha \in \CC[\fh]^W\simeq \CC[\fg]^G$$
 gives the equation of $X_{\ad}(\fg)^\vee$, where $R_\ell\subset R$ is the subset of long roots
 (with the usual convention that in a simply laced root system, all the roots are long).
 \end{prop} 
 
 A first consequence is the following:

\begin{prop}
\label{proppossiblex_simply_laced}
Suppose that $\fg$ is simply laced. Let $x\in\fg$ be such that the hyperplane section $H_x=X_{\ad}(\fg)\cap \PP(x^\perp)$
is smooth. Then $x$ is regular semisimple.
\end{prop}

\begin{proof}
Let $x=x_s+x_n$ be the Jordan decomposition of $x$, with $x_s\in\fh$ and $[x_s,x_n]=0$. By Tevelev's formula $H_x$ is smooth if and only if $P_\fg (x)\ne 0$. Since $P_\fg$ is $G$-invariant this is equivalent to $P_\fg(x_s)\neq 0$, which implies that 
the centralizer of $x_s$ is exactly $\fh$, so that $x_s$ is regular. Then from $[x_s,x_n]=0$ we deduce that in fact
$x_n=0$. 
\end{proof}

In the non simply laced case the situation is richer.

\begin{prop}
\label{classification_possible_h}
Suppose that $\fg$ is not simply laced nor of type C. Let $x\in\fg$ be such that the hyperplane section $H_x=X_{\ad}(\fg)\cap \PP(x^\perp)$
is smooth. Consider the Jordan decomposition $x=x_s+x_n$, and suppose that $x_s\in\fh$. Then one of the following holds:
\begin{enumerate}
\item $x_s$ is regular, and therefore $x_n=0$;
\item $x=x_s$ and there is a unique short positive root $\beta$ such that $\beta(x_s)=0$;
\item there is a unique short positive root $\beta$ such that $\beta(x_s)=0$,  and we may suppose that $x_n\neq 0$ 
belongs to the root space $\fg_\beta$;
\item $x=x_s$ and there exist two  short positive roots $\beta, \beta'$ such that $\beta(x_s)=\beta'(x_s)=0$, 
and they generate a subroot system of type $A_2$;
\item there exist two  short positive roots $\beta, \beta'$ such that $\beta(x_s)=\beta'(x_s)=0$, 
and they generate a subroot system of type $A_2$; moreover $x_n=X_\beta$;
\item there exist two  short positive roots $\beta, \beta'$ such that $\beta(x_s)=\beta'(x_s)=0$, 
and they generate a subroot system of type $A_2$; moreover $x_n=X_{\beta}+X_{\beta'}$.
\end{enumerate}
\end{prop}

\begin{proof}
The proof is similar to the one of Proposition \ref{proppossiblex_simply_laced}. The only difference is that in this case $x_s$ is not required to be regular, but it can be orthogonal to some short roots whose span does not contain any long root; then these roots generate a root subsystem consisting only of short roots. The claim of the proposition is then a consequence of the fact that the root subsystems of the root systems of type B,F,G consisting only of short roots are of type $A_1$ or $A_2$.
\end{proof}

The last three possibilities can only occur in type $F_4$, for which we get exactly six cases up to conjugation. Indeed there is a unique short root subsystem 
of type $A_2$ up to the Weyl group action. Let us denote by $\beta,\beta'$ the two simple roots of this subsystem, and by $X_{\bullet}$ a non-zero vector in the root space $\fg_{\bullet} $. Then the six cases are given by $x=x_s+x_n$ with: 
$$\begin{array}{cccc}
 x_s & x_n&& \\ 
 \mathrm{regular\;semisimple} &  0,&& \\
 Ker\beta & 0,& X_\beta&  \\
 Ker\beta\cap Ker\beta' & 0,& X_\beta,&  X_\beta +X_{\beta'}.
  \end{array}$$

\subsection{The connected component} 
\label{sec_connected_component}

We already have enough information to determine the connected component of $\Aut(H_x)$, since 
$$\Aut^0(H_x)\simeq N_G^0(\CC x),$$ where the RHS is the connected component of the identity of the normalizer of $\CC x$ inside $G$.
Indeed this reduces to 
Lie algebra computations, since $\Aut^0(H_x)\subset G$ is uniquely determined by its Lie algebra 
$\mathfrak{n}_\fg( \CC x)$. According to the six cases of Proposition \ref{classification_possible_h},
we obtain the following possibilities. 

\medskip\noindent {\it First case:} $x\in \fh$ is regular semisimple, then  $\mathfrak{n}_\fg(\CC x)=\fh$ and $\Aut^0(H_x)=T$. 

\smallskip\noindent {\it Second case:} Here $G$ is of type $B_n, G_2, F_4$ and $x\in \fh$ is semisimple with $(x,\alpha_n)= 0$ for $\alpha_n$ a short simple root, and $(x,\alpha)\neq 0$ for all roots $\alpha\neq \alpha_n$. Then
\[
\mathfrak{n}_{\fg}(\CC  x)=\fh \oplus \fg_{\alpha_n}\oplus \fg_{-\alpha_n}.
\]
and $\Aut^0(H_x)$ is an extension of a type $A_2$ group by the torus corresponding to $\alpha_n^\perp$. 

\smallskip\noindent {\it Third  case:} 
Here again $G$ is of type $B_n, G_2,F_4$ and $x=x_s+x_n$, with $x_n\neq 0$ and $x_s\subset \fh$ subregular. 
In particular we can assume that $x_n\in \fg_{\alpha_n}$. We get
$\mathfrak{n}_{\fg}(\CC x)=\alpha_n^\perp \oplus \fg_{\alpha_n}$ with $\alpha_n^\perp \subset \fh$, so that $\Aut^0(H_x)$ is an extension of the additive group $\GG_a$ by the torus corresponding to $\alpha_n^\perp$. 

\smallskip\noindent {\it Fourth  case:} Here $G$ is of type $F_4$, $x\in\fh$ is semisimple and orthogonal to the two short simple roots $\alpha_3,\alpha_4$ (and no other root). Then
\[
\mathfrak{n}_{\fg}(\CC x)=\langle\alpha_3,\alpha_4\rangle^\perp \oplus \fsl_3(\alpha_3,\alpha_4).
\]
where $\fsl_3(\alpha_3,\alpha_4)$ is the $\fsl_3$ subalgebra generated by the root spaces of  $\pm\alpha_3$
 and $\pm\alpha_4$. Then $\Aut^0(H_x)$ is an extension of a type $A_3$ group by the torus corresponding to $\langle\alpha_3,\alpha_4\rangle^\perp$. 

\smallskip\noindent {\it Fifth  case:} Here again
$G$ is of type $F_4$, $x=x_s+x_n$ with $x_s\in \fh$ semisimple and orthogonal to the two short simple roots $\alpha_3,\alpha_4$ (and no other root) and we can assume that $0\neq x_n=X_{\alpha_3+\alpha_4}\in \mathfrak{g}_{\alpha_3+\alpha_4}$. Then 
\[
\mathfrak{n}_{\fg}(\CC x)=(\alpha_3+\alpha_4)^\perp \oplus \fg_{\alpha_3}\oplus \fg_{\alpha_4}\oplus \fg_{\alpha_3+\alpha_4}
\]
and $\Aut^0(H_x)$ is an extension of a three-dimensional unipotent group by the torus corresponding to $(\alpha_3+\alpha_4)^\perp$. 

\smallskip\noindent {\it Sixth  case:}  Here again
$G$ is of type $F_4$, $x=x_s+x_n$ with $x_s\in\fh$ semisimple and orthogonal to the two short simple roots $\alpha_3,\alpha_4$ (and no other root) and we can assume that $x_n=X_{\alpha_3}+X_{\alpha_4}$. Then 
\[
\mathfrak{n}_{\fg}(\CC x)=\langle\alpha_3,\alpha_4\rangle^\perp \oplus \fg_{\alpha_3+\alpha_4}\oplus \CC x_n
\]
and $\Aut^0(H_x)$ is an extension of a two-dimensional unipotent group by the torus corresponding to $\langle\alpha_3,\alpha_4\rangle^\perp$. 

\medskip Note that if $r$ is the rank of $\fg$, the dimension of $\Aut^0(H_x)$ is $r+\epsilon$, 
where according to the six previous cases we have $\epsilon=0,2,0,6,2,0$. Also note that the six cases 
give each only one conjugacy class of subgroups of $G$. 

\begin{remark} Notice that when  $x$ is not semisimple, by 
Matsushima's theorem \cite{matsu} the Fano variety $H_x$ cannot admit any K\"ahler-Einstein metric, since 
its automorphism group is not reductive.

Beyond adjoint varieties, this criterion applies to a couple of other cases like hyperplane sections of $G(2,5)$,  or of the spinor tenfold $OG(5,10)$. Indeed, we know that in these cases the automorphism group of a smooth hyperplane section $H_x$ of 
our variety $X\subset\PP V$ is just the stabilizer of $\CC x$ in the group $G=Aut(X)$. By Matsushima's theorem this 
stabilizer is reductive if and only if the $G$-orbit of $\CC x$ in $\PP V$ is affine. But here this $G$-orbit  is
simply the complement of $X$, and is not affine since $X$ is not a divisor. 
\end{remark}

Let us now turn to the component group of 
 $\Aut(H_x)$, which can vary in a more subtle way. In the next two subsections we will describe 
 $\pi_x:=N_G([x])/N_G^0([x])$, and in the last one we will discuss the possible contribution of the 
 outer automorphism group $\Gamma$.

\subsection{Regular semisimple elements}
If $x=x_s\in\fh$ is regular, its centralizer is $T$ and the normalizer of the line $[x]$ must 
be contained in the normalizer of $T$. We are thus reduced to computing its  normalizer inside the Weyl group. 
This normalizer has been completely described by Springer \cite{Springer},
whose results we will summarize in this section. A first observation is:

\begin{lemma}
Let $x$ be a regular semisimple element. Then $\pi_x$ is equal to $\Stab_{W}([x])$ and it is a cyclic group.
\end{lemma}

\begin{proof}
By \cite[Proposition 4.1]{Springer}), $x$ is regular if and only if its stabilizer inside $W$ is trivial. Therefore $\Stab_{W}([x])$ acts effectively on $\CC x$, hence the statement.
\end{proof}

Recall that the algebra of invariants $\CC[\fg]^{G}$ is the same as $\CC[\fh]^{W}$, and is therefore a polynomial algebra.
We fix a basis $f_1,\ldots,f_r$ of $\CC[\fh]^{W}$ of algebraically independent homogeneous elements, where $r$ is the rank of $\fg$, 
and we denote their degrees by $d_1,\ldots , d_r$. The corresponding hypersurfaces will be denoted $F_i=V(f_i)\subset 
\PP(\fh)$. For any positive integer $d$, let 
\[
V_d:= \bigcap_{d \nmid d_i}F_i.
\]
Since $\bigcap_{i=1}^rF_i=0$, the codimension of $V_d$ is equal to $a(d)$, where 
\[
a(d):=\# \{ d_i \mbox{ such that }d\nmid d_i\}.
\]
The stratification given by the $V_d$'s is closely related to the stabilizers in the Weyl group. Indeed, if $\xi$ is a $d$-th root of unity, for any element $w\in W$ let us define
\[
V(w,\xi):= \{ v\in \fh \mid w(v)=\xi v \}.
\]
The following result can be found in \cite[Proposition 3.2, Theorem 3.4, Theorem 4.2]{Springer}.

\begin{theorem}
\label{thmspringer}
Let us fix a primitive $d$-th root of unity $\xi$. Then
\[
V_d=\bigcup_{w\in W} \PP(V(w,\xi)) \quad and \quad 
\max_{w\in W} \dim(V(w,\xi))=a(d).
\]
If $v\in V(w,\xi)$ is a regular element then 
\begin{itemize}
\item[i)] the order of $w$ is $d$;
\item[ii)] $\dim(V(w,\xi))=a(d)$;
\item[iii)] the elements $w$ in $W$ satisfying $ii)$ form a single conjugacy class.
\end{itemize}
\end{theorem}

An element $w\in W$ is {\it regular} if it admits a regular eigenvector. For such a $w$, its order will be called a {\it regular number}. The sets $E_\#$ of regular numbers of simple Lie algebras were computed in \cite{Springer}
and are reported in Table \ref{tableregularnumbers}.

\begin{table}[]
\centering
\begin{tabular}{c|c|c}
Type & $d_1,\ldots, d_n$ & $E_\#$ \\
\hline
 $A_n$ & $2,3,\ldots,n+1$ & $\{d\geq 2 \mbox{ such that }d\mid n \mbox{ or }d\mid n+1\}$ \\
 $B_n$, $C_n$ & $2,4,6,\ldots,2n$ & $\{d\geq 2 \mbox{ such that }d\mid 2n\}$ \\
 $D_n$ & $2,4,6,\ldots,2n-2,n$ & $\{d\geq 2 \mbox{ such that }d\mid 2n-2 \mbox{ or }d\mid n\}$ \\    
 $G_2$ & $2,6$ & $2,3,6$ \\
 $F_4$ & $2,6,8,12$ & $2,3,4,6,8,12$ \\
 $E_6$ & $2,5,6,8,9,12$ & $2,3,4,6,8,9,12$ \\
 $E_7$ & $2,6,8,10,12,14,18$ & $2,3,6,7,9,14,18$ \\
 $E_8$ & $2,8,12,14,18,20,24,30$ & $2,3,4,5,6,8,10,12,15,20,24,30$
\end{tabular}

\medskip
\caption{Fundamental degrees and regular numbers}
\label{tableregularnumbers}
\end{table}

\begin{coro}
Let $x$ be a regular element in $\fh$. Then $\Stab_{W}([x])=\ZZ_d$, for $d$ the maximal integer in $E_\#$ such that $[x]$
belongs to $V_d$.
\end{coro}

\begin{coro}\label{corocyclicgroups}
The list of possible stabilizers in $W$ of regular elements in $\fg$ is given by all the groups $\ZZ_d$ with $d$ maximal inside $E_\#$ with respect to the partial ordering: $m\prec m'$ if $m\mid m'$ and $a(m)=a(m')$.
\end{coro}

\begin{proof}
If $m\prec m'$ then $V_m=V_m'$ and therefore, if the stabilizer of a regular element contains $\ZZ_m$ then it also contains $\ZZ_{m'}$. If $d$ is maximal for the partial ordering then the stabilizer of a general regular element in $V_d$ is equal to $\ZZ_d$. More precisely for any $d'$ such that $d \mid d'$, the variety $V_{d'}$ is a strict closed subvariety of $V_d$; outside all those subvarieties $V_{d'}$, a regular element in $V_d$ has stabilizer equal to $\ZZ_d$.
\end{proof}

We finally obtain the list of all possible degrees $d\ge 2$ of non trivial stabilizers of regular elements, 
as reported in Table \ref{tab:orderofstabilizers}.

\begin{table}[t]
\centering
\begin{tabular}{c|c}
Type   & degrees \\
\hline
 $A_n$ &  $d\mid n \mbox{ or }d\mid n+1$ \\
 $B_n$, $C_n$ & $d\mid 2n$, $d$ even \\
 $D_n$ & $d\mid n \mbox{ or }d\mid n-1$, $d$ even \\    
 $G_2$ & $2,6$  \\
 $F_4$ & $2,6,8,12$ \\
 $E_6$ & $2,4,6,8,9,12$ \\
 $E_7$ & $2,6,14,18$ \\
 $E_8$ & $2,4,6,8,10,12,15,20,24,30$
\end{tabular}

\medskip
\caption{Orders of stabilizers}
\label{tab:orderofstabilizers}
\end{table}

\subsection{General case}\label{sec_generalcase} 
In this section we show that even when $x$ is not regular semisimple, we can use 
Springer's results to compute the normalizer of $[x]$ in the adjoint group $G$. For this we will need to restrict to 
a smaller root system. We will need to introduce the following notations.
\begin{itemize}
\item Let $\fj$ denote the subspace of $\fh$ generated by the roots $\alpha\in R$ such that $(x_s,\alpha)=0$, and 
let $S(\fj)$ denote the subgroup of $G$ of type $A_{\dim(\fj)}$ defined by the root subsystem $R\cap \fj$.
\item Let $\fk$ denote the  orthogonal of $\fj$ inside $\fh$, with the corresponding subtorus $T^\perp:=\exp(\fk)\subset T=\exp(\fh)$, and the corresponding root subsystem $R^\perp=R\cap \fk$ with Weyl group  $W^\perp$. These definitions make sense because in this situation it turns out, after a case by case inspection, that the orthogonal of the root subsystem $R\cap \fj$ is indeed a root system inside $\fk$.
\end{itemize}

The rank of $R^\perp$ is the dimension of $\fk$.
Of course when $x$ is regular semisimple, $\fj=0$ and $R^\perp=R$. Otherwise, if $x_s$ is as in cases (2-3) of Proposition \ref{classification_possible_h}, then $R$ is of type $G_2, F_4$ or $B_n$ with $n\geq 3$ and  $R^\perp$ is of type $A_1, B_3$ or $B_{n-1}$, respectively. If $x_s$ is as in cases (4-5-6) of Proposition \ref{classification_possible_h},  $R$ is of type $F_4$ and $R^\perp$ is of type $A_2$.

\begin{lemma} $\pi_x=N_G([x])/N_G^0([x])$ embeds inside $W^\perp=\Aut(R^\perp)$.
\end{lemma}

\begin{proof}
Let us deal first with $\pi_x$ when we are not in cases (4-5-6) of Proposition \ref{classification_possible_h}. 
By the unicity of the Jordan decomposition, any $g\in N_G([x])$ stabilizes $[x_s]$ and (when $x_n\ne 0$) $[x_n]$. 
Let us show that $N_G([x_s])\subset W^\perp \ltimes (T^\perp \times S(\fj))$. 

The analysis of its Lie algebra made in Section \ref{sec_connected_component} implies  that $N_G^0([x_s])= T^\perp \times S(\fj)$.
We can therefore identify $\fk$ with the subspace of elements in $\mathfrak{g}$ on which $N_G^0([x_s])$ acts trivially. 
Since $N_G^0([x_s])$ is a normal subgroup of $N_G([x_s])$, this implies that $\fk$, and therefore $T^\perp$ as well,
are stabilized by $N_G([x_s])$. Similarly, we can identify $\fk\oplus \fsl(\fj)$ with the subspace of elements in $\mathfrak{g}$ on which $T^\perp$ acts trivially. So again $N_G([x_s])$ preserves $\fk\oplus \fsl(\fj)$, and since 
it preserves $\fk$ it also preserves its orthogonal complement with respect to the Killing form, namely $\fsl(\fj)$.
But this is $0$ or $\fsl_2$, so modulo the action of $S(\fj)$ we can suppose that the action of $N_G([x_s])$ on 
$\fsl(\fj)$ is trivial. Then it stabilizes $\mathfrak{h}$, so embeds inside  $W\ltimes T$.
Since it also fixes $\fj$, the image is contained in $W^\perp\ltimes T^\perp$, so  
\[
N_G([x_s])\subset W^\perp \ltimes (T^\perp \times S(\fj)).
\]
To conclude, notice that for any $x_n\ne 0$ in Proposition \ref{classification_possible_h} we have 
\[
\Stab_{S(\fj)}([x_n])=\Stab^0_{S(\fj)}([x_n]), 
\]
from which we deduce that 
\[
N_G([x])\subset W^\perp \ltimes (T^\perp\times \Stab^0_{S(\fj)}([x_n])).
\]
Killing the connected component then yields the result. 

Now suppose that $x_s$ is as in cases (4-5-6) of Proposition \ref{classification_possible_h}, so that $R$ is of type $F_4$, 
$\dim(\fj)=2$ and $R\cap \fj$ is of type $A_2$. In order to follow the same arguments as before we just need to ensure
that any automorphism of  $\fsl(\fj)$ can be lifted to $N_G^0([x_s])$. This is clear for the inner automorphisms, and the 
outer automorphism is taken care by the next Lemma, which therefore concludes this proof. 
\end{proof}

\smallskip
Let $W$ be the Weyl group of $F_4$ and $R$ its root system. If $\alpha_3,\alpha_4\in R$ are the simple short roots and $\alpha_1,\alpha_2\in R$ the simple long roots, then $\langle \alpha_3,\alpha_4 \rangle^\perp=\langle \alpha_1,\beta\rangle$, where $\beta=\alpha_1+3\alpha_2+4\alpha_3+2\alpha_4$. 
Both root subsystems $R\cap \langle \alpha_3,\alpha_4 \rangle$ and $R\cap \langle \alpha_3,\alpha_4 \rangle^\perp$ are of type $A_2$.

\begin{lemma}
\label{lemouterF4}
Modulo the  Weyl groups of $R\cap \langle \alpha_3,\alpha_4 \rangle$ and $R\cap \langle \alpha_3,\alpha_4 \rangle^\perp$, there exists a unique $\sigma\in W$ acting as an outer isomorphism on both root subsystems.
Moreover no element of $W$ can act as an outer automorphism on one of the two root subsystems, and as an inner automorphism on the other.
\end{lemma}

\begin{proof}
If $\sigma$ does not act as an outer automorphism on $\langle \alpha_3,\alpha_4 \rangle$, we can assume that it acts on this subspace as the identity by composing with an element of the $A_2$ Weyl group of $R\cap \langle \alpha_3,\alpha_4 \rangle$. Moreover an outer automorphism of $\langle \alpha_3,\alpha_4 \rangle^\perp$ modulo the $A_2$ Weyl group of $R\cap \langle \alpha_3,\alpha_4 \rangle^\perp$ acts by exchanging $\alpha$ and $\beta$. One can check that $\sigma(\alpha_2)$ is not a root, thus such a $\sigma$ cannot exist.
An analogous argument applies if $\sigma$ does not act as an outer automorphism on $\langle \alpha_3,\alpha_4 \rangle^\perp$.

If $\sigma$ acts as an outer automorphism both on $\langle \alpha_3,\alpha_4 \rangle$ and $\langle \alpha_3,\alpha_4 \rangle^\perp$, modulo the action of the two Weyl groups stabilizing the two subspaces, $\sigma$ exchanges $\alpha_3$ with $\alpha_4$ and $\alpha_1$ with $\beta$. One can check that $\sigma:=s_{\alpha_2}\circ s_{\alpha_3}\circ s_{\alpha_2}\circ s_{\alpha_4}\circ s_{\alpha_3}\circ s_{\alpha_2}$ does the trick.
Moreover, since $\langle \alpha_1,\alpha_3,\alpha_4,\beta\rangle=\fh$, the condition that $\sigma$ must exchange $\alpha_3$ with $\alpha_4$ and $\alpha_1$ with $\beta$ (modulo the action of the Weyl groups) defines it uniquely as a linear automorphism of $\fh$, hence in $W$.
\end{proof}

Now we can deduce the statement we are aiming at.

\begin{prop}
\label{propgroupB}
The finite group $\pi_x=N_G([x])/N_G^0([x])$ is given by 
$$\pi_x\simeq \Stab_{W^\perp}([x_s])\rtimes B_x,$$
where $B_x=\ZZ_2$ in cases (4-5-6) of Proposition \ref{classification_possible_h} if $\alpha_1(x)=\alpha_2(x)$ for $\alpha_1,\alpha_2 \in R^\perp$ which are $W^\perp$-conjugate to the simple roots of $R^\perp=A_2$; and 
$B_x=1$ otherwise.
\end{prop} 

\begin{proof}
When $x_n$ is nonzero the group $\Gamma$ is trivial. In cases (2-3)  of Proposition \ref{classification_possible_h} we can identify $\pi_x$ with $\Stab_{W^\perp}([x_s])$ because $(W^\perp)(x_n)=x_n$ and we can always rescale $x_n$ by  
an element of $G$ fixing $x_s$ in order to fix $x$.
t
In cases  (4-5-6) of the Proposition, the nilpotent element $x_n$ can be chosen among $0,X_{\alpha_3+\alpha_4}$ or $X_{\alpha_3}+X_{\alpha_4}$, so that it is fixed by the Cartan involution 
exchanging $\alpha_3$ and $\alpha_4$. Moreover a conjugate of the element $\sigma$ appearing in Lemma \ref{lemouterF4} acts by exchanging $\alpha_1$ and $\alpha_2$ on $\fk$ and fixes $x_n\in \fsl(\fj)$. By rescaling $x_n$ as before, we deduce that a conjugate of $\sigma$ fixes $[x]$ and 
thus descends to an automorphism of $H_x$, giving the extra $\ZZ_2$. 
\end{proof}

As we already mentioned we can then describe $\Stab_{W^\perp}([x_s])$ with the help of Springer's result 
summarized in the previous subsection. 

\medskip\noindent {\it Example}. Let us describe $\pi_x$ when $\fg$ is of type $F_4$. If $x$ is regular then $\pi_x$ is either $\ZZ_2,\ZZ_6,\ZZ_8$ or $\ZZ_{12}$ (see Table \ref{tab:orderofstabilizers}); if $x$ is subregular then $R^\perp$ is of type $B_3$ and $\pi_x$ is either $\ZZ_2$ or $\ZZ_6$ (see Table \ref{tab:orderofstabilizers}); if $x$ is as in case (4-5-6) of Proposition \ref{classification_possible_h} then $R^\perp$ is of type $A_2$ and $\pi_x$ is either $\ZZ_2$ or $\ZZ_3$ (if $B_x=1$, see Table \ref{tab:orderofstabilizers}), $\ZZ_2\rtimes \ZZ_2$ or $\ZZ_3\rtimes \ZZ_2$ (in the hypothesis of case $i)$ of the previous proposition, i.e. when $B_x=\ZZ_2$).

\subsection{The action of outer automorphisms}

In order to completely describe $\Aut(H_x)$, there remains to understand the automorphisms coming from $\Gamma$. If the 
latter is non trivial, we are either in type $A_n$ ($n\geq 2$), $D_n$ ($n\geq 4$) or $E_6$, and  $\Gamma=\ZZ_2$
except in type $D_4$ for which  $\Gamma\simeq\fS_3$. Recall moreover that since we are in the simply laced case, $x$ must be 
regular semisimple. 

Denote by $w_0$ the longest element of $W$. Since $w_0\circ \Gamma\circ w_0=\Gamma$, if $\Gamma\cong \ZZ_2$ then either $\Gamma\circ w_0=\pm\id|_{\fh}$ or $w_0=-\id|_{\fh}$, the first possibility occurring exactly in type $D_{2m+1}$ and $E_6$ (see \cite{bourbaki}, tables at the end of Chapter 6).

Let $G$ be of type $D_{2m}$, $m\geq 3$. The root system orthogonal to any root $\alpha$ is of type $D_{2m-2}\times A_1$. Indeed, any root can be sent to the $n$-th root of the Dynkin diagram of $D_{2m}$ by an element $w\in W$. The group $w^{-1}\circ \Gamma \circ w$ acts on the set of roots by exchanging $\alpha$ with one of the two roots of the orthogonal root system $A_1$. Any conjugate of $\Gamma$ acts in this way.

If $G$ is of type $D_4$, the root system orthogonal to any root is of type $A_1\times A_1\times A_1$ and a certain conjugate of $\Gamma$ acts by exchanging the three $A_1$ factors.

Let us denote by $\tilde{\pi}_x:=\Aut(H_x)/\Aut^0(H_x)$. At this point, we know that, when $\Gamma$ is not trivial, $\tilde{\pi}_x$ is an extension by $\Stab_W([x])$ of the group $C_x$ of elements of $\Gamma$ that descend to $H_x$.

\begin{prop}
\label{lemgammadescends}
The group $C_x$ is non trivial exactly in the following cases:
\begin{itemize}
\item $G$ is of type $D_{2m+1}$ ($m\geq 2$) or $E_6$, and then $C_x\simeq \ZZ_2$.
\item $G$ is of type $A_n$ ($n\geq 2$) and  $(x,\alpha_i)=(x,\alpha_{n-i+1})$ for some set $\{\alpha_1,\ldots , \alpha_n\}$ of roots which is $W$-conjugate to the set of simple roots; then $C_x\simeq \ZZ_2$.
\item $G$ is of type $D_{2m}$ ($m\geq 3$), and $(x,\alpha)=(x,\alpha')$ for  two roots $\alpha, \alpha'$ such that  $\langle\alpha,\alpha'\rangle^\perp$ is of type $D_{2m-2}$; then $C_x\simeq \ZZ_2$.
\item $G$ is of type $D_4$, and $(x,\alpha)=(x,\alpha')=(x,\alpha'')$ for three mutually orthogonal roots $\alpha,\alpha',\alpha''$; then $C_x\simeq \fS_3$. 
\item $G$ is of type $D_4$, the previous condition is not satisfied but there exist two orthogonal roots 
$\alpha,\alpha'$ such that $(x,\alpha)=(x,\alpha')$; then $C_x\simeq \ZZ_2$.
\end{itemize}
\end{prop}

\begin{proof}
The first statement follows by noticing that $[x]$ is fixed by $\Gamma\circ w_0$ (which sends $x$ to $-x$). If $G$ is of type $D_{2m}$ with $m\geq 3$,  suppose that $w\in W$ sends $\alpha$ to the $n$-th simple root, and $\alpha'$ to the $(n-1)$-th simple root of the root system. Then $w^{-1}\circ \Gamma \circ w$ fixes $x$ and thus $C_x\simeq \ZZ_2$. The cases of $D_4$ and $A_n$ can be treated similarly. 
\end{proof}

\medskip\noindent {\it Example}. Let us describe completely $\Aut(H_x)$ when $\fg$ is of type $E_6$. Since $E_6$ is simply connected, $x$ is regular and $\Aut^0(H_x)=T$. Then, depending on the stabilizer of $[x]$ inside $W$, $$\Aut(H_x)\cong (\Aut^0(H_x)\rtimes \Stab_W([x]))\rtimes \Gamma\cong(T\rtimes \ZZ_d)\rtimes \ZZ_2$$ for $d=2,4,6,8,9,12$ (see Table \ref{tab:orderofstabilizers}). 

\subsection{Conclusion}
Let us synthetize the results we have obtained. Using the notations of the previous sections, the automorphism group of a smooth hyperplane section $H_x$ of an adjoint variety  can be decomposed as
 $$\Aut(H_x)=N_G([x])\rtimes C_x,$$
 where $C_x=\Stab_\Gamma([x])\subset \Gamma$ (a subgroup of the outer automorphism group of the root system $R$) 
 is described in Proposition \ref{lemgammadescends}. We showed in Section \ref{sec_connected_component} that $\Aut^0(H_x)=N_G^0([x])$ and we described this group explicitly. By Proposition \ref{propgroupB}, the quotient $\pi_x=N_G([x])/N_G^0([x])$
 can be decomposed as $$\pi_x=\Stab_{W^\perp}([x_s])\rtimes B_x,$$ 
 where $W^\perp\subset W$ is described in Section \ref{sec_generalcase}.
 In the same Proposition $B_x$ is made explicit, while $\Stab_{W^\perp}([x_s])$ is a cyclic group completely described in Corollary \ref{corocyclicgroups}.

Notice on the one hand that if $\Gamma$  is non trivial, then $x$ is regular semisimple and $W=W^\perp$; on the other hand $B_x$ is non trivial only if $W^\perp\subsetneq W$, and it is essentially the group of outer automorphisms of $R^\perp$ that are contained in $G$ and stabilize $[x_s]$. Thus $B_x$ and $C_x$ play similar roles and cannot be both non trivial;  we 
denoted their product by $D_x$ in Theorem \ref{thm_intro_adjoint}. Finally, $Aut(H_x)$ can be decomposed as
\[
\Aut(H_x)=\Aut^0(H_x)\rtimes \Stab_{W^\perp}([x_s]) \rtimes D_x.
\]
In Table \ref{final_table} we reported all the possibilities for each term of this decomposition, with the following notation:
\begin{itemize}
\item $T$ is a maximal torus in $G$, $T_{V^\perp}$ is the torus whose Lie algebra is $V^\perp \subset \fh$ for a subspace $V\subset \fh$ and $\fh$ a Cartan subalgebra of $\fg$;
\item $\GG_a$ is an additive one dimensional group, $U_2$ (respectively $U_3$) is a two (resp. three) dimensional unipotent group, $A_i$ is a Lie group of type $A_i$.
\end{itemize} 
The second column follows the notation of Proposition \ref{classification_possible_h}. In type $A$ we reported only the subgroup of $\Aut(H_x)$ 
that extends to $\Aut(X_{\ad}(\fsl_{n+1}))$, since there could be some unnatural automorphisms (see Theorem \ref{thm_exact_sequence_aut_jordan} and the discussion that follows).

\begin{table}[]
\centering
\begin{tabular}{c|c|c|c|c}
Type & $x$ & $\Aut^0(H_x)$ & $\Stab_{W^\perp}([x_s])\cong \ZZ_d$ & $D_x$ \\
\hline
 $A_n$ & (1) & $T$ & $d\geq 2 \mbox{ , }d\mid n \mbox{ or }d\mid n+1$ & $1$ or $\ZZ_2$ \\
 $B_n$ & (1) & $T$ & $d\mid 2n$, $d$ even & $1$ \\
 $B_n$ & (2) & $T_{\alpha_n^\perp}\rtimes A_2$ & $d\mid 2n-2$, $d$ even & $1$ \\
 $B_n$ & (3) & $T_{\alpha_n^\perp}\rtimes \GG_a$ & $d\mid 2n-2$, $d$ even & $1$ \\
 $D_4$ & (1) & $T$ & $d=2,4$ & $1$ or $\ZZ_2$ or $\cS_3$ \\
 $D_{2m}$ & (1) & $T$ & $d\mid 2m \mbox{ or }d\mid 2m-1$, $d$ even & $1$ or $\ZZ_2$ \\  
 $D_{2m+1}$ & (1) & $T$ & $d\mid 2m+1 \mbox{ or }d\mid 2m$, $d$ even & $\ZZ_2$ \\
 $E_6$ & (1) & $T$ & $d=2,4,6,8,9,12$ & $\ZZ_2$ \\
 $E_7$ & (1) & $T$ & $d=2,6,14,18$ & $1$ \\
 $E_8$ & (1) & $T$ & $ d=2,4,6,8,10,12,15,20,24,30$ & $1$ \\
 $G_2$ & (1) & $T$ & $d=2,6$ & $1$ \\
 $G_2$ & (2) & $T_{\alpha_2^\perp}\rtimes A_2$ & $d=2$ & $1$ \\
 $G_2$ & (3) & $T_{\alpha_2^\perp}\rtimes \GG_a$ & $d=2$ & $1$ \\
 $F_4$ & (1) & $T$ & $d=2,6,8,12$ & $1$ \\
 $F_4$ & (2) & $T_{\alpha_4^\perp}\rtimes A_2$ & $d=2$ & $1$ \\
 $F_4$ & (3) & $T_{\alpha_4^\perp}\rtimes \GG_a$ & $d=2$ & $1$ \\
 $F_4$ & (4) & $T_{\langle\alpha_3,\alpha_4\rangle^\perp}\rtimes A_3$ & $d=2,3$ & $1$ or $\ZZ_2$ \\
 $F_4$ & (5) & $T_{(\alpha_3+\alpha_4) ^\perp}\rtimes U_3$ & $d=2,3$ & $1$ or $\ZZ_2$ \\
 $F_4$ & (6) & $T_{\langle\alpha_3,\alpha_4\rangle ^\perp}\rtimes U_2$ & $d=2,3$ & $1$ or $\ZZ_2$ 
 
\end{tabular}

\medskip
\caption{Automorphisms of smooth linear sections $H_x$ of adjoint varieties}
\label{final_table}

\end{table}

\bibliographystyle{alpha}

\bigskip
\noindent
\textsc{Institut de Mathématiques de Bourgogne, UMR CNRS 5584, Universit\'e de Bourgogne et Franche-Comt\'e, 9 Avenue Alain Savary, BP 47870, 21078 Dijon Cedex, France}

\noindent
\textit{Email address}: \texttt{Vladimiro.Benedetti@u-bourgogne.fr  }

\medskip
\noindent
\textsc{Institut de Mathématiques de Toulouse, UMR 5219,  Universit\'e Paul Sabatier, F-31062 Toulouse Cedex 9, France}

\noindent
\textit{Email address}: \texttt{manivel@math.cnrs.fr}

\end{document}